\documentclass[11pt]{amsart}
\usepackage{graphicx}
\usepackage{amsmath}
\usepackage{amsfonts}
\usepackage{amssymb}
\usepackage{MnSymbol}
\usepackage{rotating}
\usepackage{amssymb}
\usepackage{gastex}
\usepackage[usenames]{color}

\newtheorem{thm}{Theorem}[section]
\newtheorem{cor}[thm]{Corollary}
\newtheorem{lem}[thm]{Lemma}
\newtheorem{prop}[thm]{Proposition}
\newtheorem{defn}[thm]{Definition}

\newcommand{\abs}[1]{\left\vert#1\right\vert}

\newcommand{\B}{pseudo-nilpotent}
\def\abs#1{\ensuremath{\lvert #1\rvert}}

\begin{document}
\title[]{A description of a class of finite semigroups that are near to being Malcev nilpotent}
\author{E. Jespers and M.H. Shahzamanian}
\address{Department of Mathematics,
 Vrije Universiteit Brussel,  Pleinlaan 2, 1050
Brussel, Belgium} \email{efjesper@vub.ac.be,
m.h.shahzamanian@vub.ac.be}
\thanks{ 
\\Research partially supported by
Onderzoeksraad of Vrije Universiteit Brussel, Fonds voor
Wetenschappelijk Onderzoek (Belgium). }

\begin{abstract}
In this paper we continue the investigations on the algebraic
structure of a finite semigroup $S$ that is determined by its
associated upper non-nilpotent graph $\mathcal{N}_{S}$. The
vertices of this graph are the elements of $S$ and two  vertices
are adjacent if they generate a semigroup that is not nilpotent
(in the sense of Malcev). We introduce a class of semigroups in
which the Mal'cev nilpotent property lifts through ideal chains.
We call this the class of \B\ semigroups. The definition is such
that the global information that a semigroup is not nilpotent
induces local information, i.e. some two-generated subsemigroups
are not nilpotent. It turns out that a finite monoid (in
particular, a finite group) is \B\ if and only if it is nilpotent.
Our main result is a description of \B\ finite semigroups $S$ in
terms of their associated graph ${\mathcal N}_{S}$.  In
particular, $S$ has a largest nilpotent ideal, say $K$, and $S/K$
is a $0$-disjoint union of its connected components (adjoined with
a  zero) with at least two elements.
\end{abstract}

\maketitle

\section{Introduction}\label{pre}


For a semigroup $S$ with elements $x,y,z_{1},z_{2},\ldots $ one
recursively defines two sequences
$$\lambda_n=\lambda_{n}(x,y,z_{1},\ldots, z_{n})\quad{\rm and}
\quad \rho_n=\rho_{n}(x,y,z_{1},\ldots, z_{n})$$ by
$$\lambda_{0}=x, \quad \rho_{0}=y$$ and
$$\lambda_{n+1}=\lambda_{n} z_{n+1} \rho_{n}, \quad \rho_{n+1}
=\rho_{n} z_{n+1} \lambda_{n}.$$ A semigroup is said to be
\textit{nilpotent} (in the sense of Mal'cev~\cite{malcev}, denote
$(MN)$ in~\cite{Riley}) if there exists a positive integer $n$ such
that
$$\lambda_{n}(a,b,c_{1},\ldots, c_{n}) = \rho_{n}(a,b,c_{1},
\ldots, c_{n})$$ for all $a, b$ in $S$ and $ c_{1}, \ldots, c_{n}$
in $S^{1}$. The smallest such $n$ is called the nilpotency class
of $S$. It is well known that a group $G$ is nilpotent of class
$n$ if and only if it is nilpotent of class $n$ in the classical
sense. Nilpotent semigroups and their semigroup algebras have been
investigated in~\cite{Eric}. For example, it is proved that a
completely $0$-simple semigroup $S$ is nilpotent if and only if
$S$ is an inverse semigroup with  nilpotent maximal subgroups. If
$S$ is a semigroup with a zero $\theta$, then obviously an ideal
$I$  with $I^n = \{\theta \}$ is a nilpotent semigroup as well.
Several consequences of the (Mal'cev) nilpotence  have appeared in
the literature. For example, in~\cite{neu-tay} a semigroup $S$ is
said to be \textit{Neumann-Taylor} $(NT)$ if there exists a
positive integer $n\geq 2$ such that
$$\lambda_n(a, b, 1, c_2, \ldots, c_n)=\rho_n(a, b, 1, c_2, \ldots, c_n)$$
for all $a,b,c_2,\ldots,c_n$ in $S^1$. A semigroup $S$ is said to
be \textit{positively Engel} $(PE)$ if for some positive integer
$n\geq 2$,
$$\lambda_{n}(a,b,1,1,c,c^{2},\ldots, c^{n-2})
=\rho_{n} (a,b,1,1,c,c^{2},\ldots, c^{n-2})$$ for all $a, b$ in
$S$ and $c\in S^{1}$. A semigroup $S$ is said to be \textit{weakly
Mal'cev nilpotent} $(WMN)$ if for some positive integer $n$,
$$\lambda_{n}(a,b,c_1,c_2,c_1,c_2,\ldots )
=\rho_{n} (a,b,c_1,c_2,c_1,c_2,\ldots )$$ for all $a, b$ in $S$
and $c_1, c_2\in S^{1}$. These classes of semigroups have been
studied in~\cite{Riley}.

In~\cite{Jes-shah} we initiated the investigations on the upper
non-nilpotent graph ${\mathcal N}_{S}$ of a finite semigroup $S$.
Recall that the vertices of ${\mathcal N}_{S}$ are the elements of
$S$ and there is an edge between $x$ and $y$ if the semigroup
generated by $x$ and $y$, denoted by $\langle x, y \rangle$, is not
nilpotent. Note that ${\mathcal N}_{S}$ is empty if $S$ is a
nilpotent semigroup. We state some of the results obtained. If a
finite semigroup $S$ has empty upper non-nilpotent graph then $S$
is positively Engel. On the other hand, a semigroup has a complete
upper non-nilpotent graph if and only if it is a completely simple
semigroup that is a band. The main result says that if all
connected ${\mathcal N}_{S}$-components of a semigroup $S$ are
complete (with at least two elements) then $S$ is a band that is a
semilattice of its connected components and, moreover, $S$ is an
iterated total ideal extension of its connected components.
Further, it is shown that some graphs, such as a cycle $C_{n}$ on
$n$ vertices (with $n\geq 5$), are not the upper
non-nilpotent graph of a semigroup.  Also, there is  precisely one
graph on $4$ vertices that is not the upper non-nilpotent
graph of a semigroup with $4$ elements.

In this paper we continue these investigations. We introduce a
class of semigroups in which the Mal'cev nilpotent property lifts
through ideal chains. We call this the class of \B\ semigroups. It
turns out that a finite monoid (in particular, a finite group) is
\B\ if and only if it is  nilpotent. Our main result is a
description of \B\ finite semigroup $S$ in terms of their
associated graph ${\mathcal N}_{S}$.  In particular, $S$ has a
largest nilpotent ideal, say $K$, and $S/K$ is a $0$-disjoint
union of its connected components (adjoined with a zero) with at
least two elements.

For standard notations and terminology we refer to~\cite{cliford}.

\section{\B\  semigroups}

Suppose that $I$ is an ideal of a semigroup $S$. If $S$ is
nilpotent then clearly so are the semigroups $I$ and $S / I$.
However in general the converse fails (see Example~2.3 in~\cite{Eric}). We will introduce a class of semigroups, called the
\B\ semigroups, for which the converse does hold. The definition
is motivated by the following lemma proved in~\cite{Jes-shah}.

\begin{lem} \label{finite-nilpotent}
A finite semigroup S is not nilpotent if and only if there exists
a positive integer $m$, distinct elements $x, y\in S$ and $ w_{1},
w_{2}, \ldots, w_{m}\in S^{1}$, such that $x = \lambda_{m}(x, y,
w_{1}, w_{2}, \ldots, w_{m})$, $y = \rho_{m}(x, y, w_{1}, w_{2},
\ldots, w_{m})$  (note that for the converse one does not need
that $S$ is finite).
\end{lem}

For convenience we call the empty set an ideal
of $S$ and thus, by definition, $S/\emptyset =S$.

\begin{defn}
A semigroup $S$ is said to be \B\ if $$\lambda_{t}(x,y,w_{1},\ldots, w_{t})
\neq
\rho_{t}(x,y,w_{1},\ldots, w_{t}),$$ $$(\lambda_{t}(x,y,w_{1},\ldots,
w_{t}),
\rho_{t}(x,y,w_{1},\ldots, w_{t})) =$$ $$ (\lambda_{m}(x,y,w_{1},\ldots,
w_{m}), \rho_{m}(x,y,w_{1},\ldots,
w_{m})),$$ for $x,y\in S$, $w_{1}, \ldots , w_{m}\in S^{1}$, $I$ an ideal (possibly empty) of  $\langle x,y, w_1, \ldots, w_m\rangle$,
$\lambda_m, \rho_m\not\in  I$, and non-negative integers $t < m$, implies that there are
edges in ${\mathcal N}_{\langle x,y, w_1, \ldots, w_m\rangle/I}$ between
$\lambda_i$ and $\rho_i$ for
every $0 \leq i \leq m$.
\end{defn}

So in some sense the global condition that $S$ is not nilpotent
determines local information, i.e. some two-generated
subsemigroups are not nilpotent.

Examples of \B\ semigroups are nilpotent semigroups and semigroups
with complete upper non-nilpotent graphs. Obviously, subsemigroups
and Rees factor semigroups of \B\ semigroups are also \B. Hence,
the \B\ condition yields restrictions on the completely $0$-simple
components of a semigroup. It easily can be verified that the
completely $0$-simple semigroup $\mathcal{M}^0(\{e\}, 2, 2; \left(
\begin{matrix} 1 & 1\\ \theta & 1 \end{matrix} \right))$ is not
\B.

Note that if in the definition of \B\ semigroup $S$ one would assume that  $I$ is an ideal of $S$ and require
that there are edges ${\mathcal N}_{S/I}$ between $\lambda_i$ and $\rho_i$ for
every $0 \leq i \leq m$ then subsemigroups or Rees factors of \B\ semigroups need not to inherit this condition.
Hence the requirement that $I$ is an ideal of $\langle x,y, w_1, \ldots, w_m\rangle$

\begin{lem}\label{empty}
Let $S$ be a \B\ finite semigroup. If $I$ is an ideal of $S$ such
that $I$ and  $S / I$ are nilpotent, then $S$ is nilpotent.
\end{lem}

\begin{proof}
Suppose that $S$ is not nilpotent. Then by
Lemma~\ref{finite-nilpotent}, there exists a positive integer $m$,
distinct elements $x, y\in S$ and elements $ w_{1}, w_{2}, \ldots,
w_{m}\in S^{1}$, such that $x = \lambda_{m}(x, y, w_{1},
w_{2},\ldots, w_{m})$ and $y = \rho_{m}(x, y, w_{1}, w_{2},
\ldots, w_{m})$. As $S$ is \B, this implies that there is an edge
in ${\mathcal N}_S$ between $x$ and $y$, i.e. $\langle x, y
\rangle$ is not nilpotent. If both $x$ and $y$ are not in $I$,
then none of the elements $ w_{1}, w_{2}, \ldots, w_{m},
\lambda_{i}(x,y,w_{1},\ldots, w_{i}), \rho_{i}(x,y,w_{1},\ldots,
w_{i})$  (for $1 \leq i \leq m$) belong to $I$, as $I$ is an ideal
of $S$. However, this is in contradiction with $S /I$ being
nilpotent. So $x \in I$ or $y \in I$.

The equalities $ x= \lambda_{m}(x, y, w_{1}, w_{2},\ldots, w_{m})$
and $y = \rho_{m}(x, y, w_{1}, w_{2}, \ldots, w_{m})$ imply that
there exist elements $a, b, a', b' \in S^1$ such that $x = ayb,
y=a'xb'$. Again because $I$ is an ideal of $S$, it follows that
$x, y \in I$. Since $\langle x, y \rangle$ is not nilpotent, this
yields a contradiction with $I$ being nilpotent.
\end{proof}


The \B\ condition also includes restrictions on how elements in
different principal factors multiply. Indeed, for if an ideal $I$
of a semigroup $S$ and its factor semigroup $S/I$ are \B, then it
does not necessarily follow that $S$ is also \B. For example, the
semigroup $S = \{a, b, c, d, e\}$ with multiplication table\footnote{In order to check the associativity law for the constructed examples, a software is
developed in C++ programming language.}
$$
\begin{tabular}{ c|c c c c c}
      & $a$ & $b$ & $c$ & $d$ & $e$\\ \hline
  $a$ & $a$ & $a$ & $a$ & $a$ & $a$\\
  $b$ & $a$ & $b$ & $a$ & $a$ & $e$\\
  $c$ & $a$ & $a$ & $c$ & $d$ & $a$\\
  $d$ & $d$ & $d$ & $d$ & $d$ & $d$\\
  $e$ & $e$ & $e$ & $e$ & $e$ & $e$\\
\end{tabular}$$
has an ideal $I = \{a, d, e\}$ with  $I$ and $S /I$  \B. But $S$
is not \B, because there is no edge in ${\mathcal N}_S$ between
$b$ and $e$ but
$$(\lambda_{1}(b,e,a),\rho_{1}(b,e,a))=(\lambda_{2}(b,e,a,a),\rho_{2}(b,e,a,a)).$$

In \cite{Jes-shah} an example is given of
a finite semigroup that is not nilpotent and has empty upper non-nilpotent graph. The next lemma shows that such an example can not be \B.

\begin{lem}\label{empty graph}
If a finite semigroup $S$ is \B\ and ${\mathcal N}_{S}$ is empty,
then $S$ is nilpotent.
\end{lem}

\begin{proof}
Indeed, suppose the contrary. That is,
assume $S$ is not nilpotent. Then, by Lemma~\ref{finite-nilpotent},
there exists a positive integer $m$, distinct elements $x, y\in S$
and elements $ w_{1}, w_{2}, \ldots, w_{m}\in S^{1}$, such that $x
= \lambda_{m}(x, y, w_{1}, w_{2}, \ldots, w_{m})$ and $y =
\rho_{m}(x, y, w_{1}, w_{2}, \ldots, w_{m})$. However, as $S$ is
\B, we get that $\langle x, y \rangle$ is not nilpotent. This
contradicts with ${\mathcal N}_{S} = \emptyset$.
\end{proof}

Clearly for a finite semigroup we have the following implications
$$(MN) \Rightarrow (WMN),\;
(MN) \Rightarrow (NT) \; \mbox{ and } (MN) \Rightarrow \mbox{\B
}.$$ In~\cite{Jes-shah} an example is given of a finite semigroup
$S$ which is $(NT)$ but not $(MN)$ and with ${\mathcal N}_S =
\emptyset$. So $S$ is not \B. Of course a finite semigroup
$S$ for which ${\mathcal N}_S$ is complete and ${\mathcal N}_S
\neq \emptyset$ is \B\ but it is neither $(NT)$ nor $(WMN)$. Recall
from Proposition~3.4 in~\cite{Jes-shah} that every finite
semigroup $S$ for which ${\mathcal N}_S$ is complete and $\abs{S}
> 1$, is isomorphic with a completely simple semigroup
$\mathcal{M}(\{e\},I,\Lambda;P)$ (in particular it is a band). As
all elements of $P$ are equal to $e$, we have, for all $a,b\in S$,
$$a = \lambda_{2}(a,b,1_S,1_S), b = \rho_{2}(a,b,1_S,1_S) .$$
It follows that such semigroup is neither $(NT)$ nor $(WMN)$.

Hence the classes $(WMN)$, $(NT)$ and \B\ are pairwise distinct
classes containing the Mal'cev nilpotent semigroups (Recall from Corollary 12 in \cite{Riley} that a finite semigroup is (WMN) if and only if
it is (MN).). However, in
the following lemma we show that for finite groups these notions
are the same.

\begin{lem}\label{RN Groups}
A finite group is \B\ if and only if it is nilpotent.
\end{lem}

\begin{proof}
Let $G$ be a finite group. Of course if $G$ is nilpotent then $G$
is \B. For the converse, recall that  $G$ is nilpotent if and only
if it is Neumann Taylor. So assume $G$ is \B. We need to prove
that $G$ is Neumann Taylor.  We prove this by contradiction. So
suppose $G$ is not $(NT)$. Let $k=\abs{G}$. There exist distinct
elements $a, b\in G$ and $w_{1}, \ldots, w_{k^2}$ $ \in G$ such
that
$$\lambda_{k^2+1}(a,b,1_G, w_{1},\ldots, w_{k^2})
\neq \rho_{k^2+1}(a,b,1_G, w_{1}, \ldots, w_{k^2}).$$ Since
$\abs{G\times G} =k^2$ there exist non-negative integers $t, r
\leq k^2$, $t< r$ with
 \begin{eqnarray*}
 \lefteqn{
 (\lambda_{t+1}(a,b,1_G,w_{1},\ldots, w_{t}),
\rho_{t+1}(a,b,1_G,w_{1},\ldots, w_{t})) } \\&=&
(\lambda_{r+1}(a,b,1_G,w_{1}, \ldots, w_{r}),
\rho_{r+1}(a,b,1_G,w_{1},\ldots, w_{r})).
 \end{eqnarray*}

Therefore we also have \begin{eqnarray*}
 \lefteqn{
 (\lambda_{t+1}(a,1_G,b,w_{1},\ldots, w_{t}),
\rho_{t+1}(a,1_G,b,w_{1},\ldots, w_{t})) } \\&=&
(\lambda_{r+1}(a,1_G,b,w_{1}, \ldots, w_{r}),
\rho_{r+1}(a,1_G,b,w_{1},\ldots, w_{r})).
 \end{eqnarray*}
Because $G$ is \B, it implies that there is an edge in ${\mathcal
N}_{S}$ between $1_G$ and $a$, a contradiction.
\end{proof}

Recall from~\cite{Jes-shah} that the lower non-nilpotent graph
${\mathcal L}_{S}$ of a semigroup $S$ is the graph whose vertices
are the elements of $S$ and there is an edge between two distinct
vertices $x,y\in S$ if and only if there exist elements $w_{1},
w_{2}, \ldots, w_{n}$ in $\langle x, y \rangle^{1}$ with  $x=
\lambda_{n}(x, y, w_{1}, w_{2}, \ldots, w_{n})$ and $y=
\rho_{n}(x, y, w_{1}, w_{2}, \ldots, w_{n})$. Clearly, because of
Lemma~\ref{finite-nilpotent}, ${\mathcal L}_{S}$ is a subgraph of
${\mathcal N}_{S}$. In general ${\mathcal L}_{S}$ and ${\mathcal
N}_{S}$ are different.

\begin{prop}
Let $S$ be a finite semigroup. If $S$ is \B\ and ${\mathcal
L}_{S}$ is empty, then ${\mathcal N}_{S}$ is empty.
\end{prop}

\begin{proof} Suppose
${\mathcal N}_{S}$ is not empty. Then $S$ is not nilpotent and
hence, by Lemma~\ref{finite-nilpotent}, there exists a positive
integer $m$ and distinct elements $x, y \in S$ and $w_{1}, w_{2},
\ldots, w_{m}\in S^{1}$ such that $x = \lambda_{m}(x, y, w_{1},
w_{2}, \ldots, w_{m})$ and $y = \rho_{m}(x, y, w_{1}, w_{2},
\ldots, w_{m})$. Since, by assumption, $S$ is \B,  we get that
$S_1 = \langle x, y \rangle$ is not nilpotent.

Note that the subsemigroup $S_1$ is \B. Since $S_1$ is not nilpotent,  there exists a
positive integer $m_1$, distinct elements $x_1, y_1 \in S_1$ and
$w^{'}_{1}, w^{'}_{2}, \ldots, $ $w^{'}_{m_{1}} \in S_{1}^{1}$
such that $x_1 = \lambda_{m_1}(x_1, y_1, w^{'}_{1}, w^{'}_{2},
\ldots, w^{'}_{m_1})$ and $y_1 = \rho_{m_1}(x_1, y_1, w^{'}_{1},$
$ w^{'}_{2}, \ldots, w^{'}_{m_1})$. Therefore $S_2 = \langle x_1,
y_1 \rangle$ is not nilpotent.

Let $x_{0}=x$ and $y_{0}=y$. By induction we obtain for each
non-negative integer $i$ a subsemigroup $S_i$ of $S$ such that
  $$S_{0}=S, \quad  S_i= \langle x_{i-1},
y_{i-1} \rangle   (\mbox{ for } i\geq 1)$$ and elements
$$x_i, y_i, w^{(i)}_{1}, w^{(i)}_{2}, \ldots, w^{(i)}_{m_i} \in S_{i}^{1}$$
such that
$$x_i = \lambda_{m_i}(x_i,
y_i, w^{(i)}_{1}, \ldots, w^{(i)}_{m_i}), y_i = \rho_{m_i}(x_i,
y_i, w^{(i)}_{1}, \ldots, w^{(i)}_{m_i})$$ and $$x_i \neq y_i.$$

Clearly,
$$S \supseteq S_1 \supseteq S_2 \supseteq \cdots $$
of subsemigroups of $S$. Since $S$ is finite, we get that $S_t=
S_{t+1}$ for some positive integer  $t$. Hence $x_t \neq y_t$ and
$$x_t = \lambda_{m_t}(x_t,
y_t, w^{(t)}_{1}, \ldots, w^{(t)}_{m_t}),\; y_t = \rho_{m_t}(x_t,
y_t, w^{(t)}_{1}, \ldots, w^{(t)}_{m_t}),$$
$$x_t, y_t, w^{(t)}_{1},
w^{(t)}_{2}, \ldots, w^{(t)}_{m_t} \in S_{t}^{1} = S_{t+1}^{1}=
\langle x_t, y_t \rangle.$$ So, in the graph ${\mathcal
L}_{\langle x_t, y_t \rangle}$ there is an edge between  $x_t$ and
$y_t$. Since ${\mathcal L}_{\langle x_t, y_t \rangle}$ is a
subgraph of ${\mathcal L}_{S}$, we obtain that the graph
${\mathcal L}_{S}$ is not empty, as desired.
\end{proof}

We finish this section with proving one more restriction that the
\B\ condition imposes.

\begin{lem}\label{empty ideal}
Let $S$ be a \B\ finite semigroup. If $I$ is an ideal of $S$ and
${\mathcal N}_{I}$ is empty, then elements of $I$ are isolated
vertices in ${\mathcal N}_{S}$.
\end{lem}

\begin{proof}
We prove the result by contradiction. So, suppose that there
exists an edge in ${\mathcal N}_{S}$ between $a \in S \backslash
I$ and $b \in I$. Hence $\langle a, b \rangle = \langle a\rangle
\cup B$, for some subset $B$ of $I$. Because $\langle a, b
\rangle$ is not nilpotent,  by Lemma~\ref{finite-nilpotent}, there
exists a positive integer $m$ and elements $x, y, w_1,w_2,
\ldots,w_m$ of $\langle a, b \rangle^1$ such that $x =
\lambda_m(x, y,w_1,w_2, \ldots,w_m), y =\rho_m(x, y,w_1,w_2,
\ldots,w_m)$ and $x \neq y$. As the cyclic semigroup $\langle a
\rangle$ is nilpotent, we get that $\{x, y,w_1,w_2, \ldots,w_m\}
\cap B \neq \emptyset$. Because  $I$ is an ideal, the equalities
$x = \lambda_m(x, y,w_1,w_2, \ldots,w_m)$ and $y =\rho_m(x,
y,w_1,w_2, \ldots,w_m)$ yield that $x, y \in I$. Since, by
assumption, $S$ is \B, we know that $\langle x, y \rangle$ is not
nilpotent, i.e. there is an edge in ${\mathcal N}_S$ between $x$
and $y$. However, there is contradiction with ${\mathcal N}_{I}=
\emptyset$.
\end{proof}

\section{A description of \B\ semigroups}\label{}

We begin by describing the completely $0$-simple semigroups that
are \B. The non-zero elements of a completely $0$-simple semigroup
$\mathcal{M}^0(G,I,\Lambda;P)$ over a group $G$ are denoted as
$(g; i,\lambda )$, with $g\in G$, $i\in I$ and $\lambda \in
\Lambda$. Also, we denote the sets $\{(g; i, \lambda)\mid g
\in G\}, \{(g; i, \lambda)\mid g \in G, \lambda \in \Lambda\}$ and
$\{(g; i, \lambda)\mid g \in G, i \in I\}$  by $G_{i,\lambda},
G_{i,\star}$ and $G_{\star,\lambda}$ respectively.

\begin{lem}\label{RN Rees semigroups}
Let $S$ be a \B\ finite semigroup. Assume $A$ and $B$ are ideals
of $S$ such that $A \subseteq B$ and $B / A =
\mathcal{M}^0(G,I,\Lambda;P)$ is a regular Rees matrix
semigroup. Then either $B/A$ is nilpotent (i.e. $G$ is nilpotent
and $B/A$ is an inverse semigroup) or  $\abs{I}+\abs{\Lambda}>2$,
$G$ is a nilpotent group and all entries of $P$ are non-zero (i.e. $B/A$
is a union of groups, and thus $B\backslash A$ is a semigroup).

Conversely, any completely $0$-simple semigroup
$\mathcal{M}^0(G,I,\Lambda;P)$, with $G$ a nilpotent group and all
entries of $P$ non-zero, is \B. Furthermore, if
$\abs{I}+\abs{\Lambda}>2$ and $(i,\lambda )\neq (i',\lambda ')$
then $\langle (g;i,\lambda ),\, (g';i', \lambda ')\rangle$ is not
nilpotent for all $g,g'\in G$. In particular, the upper
non-nilpotent graph of $\mathcal{M}^0(G,I,\Lambda;P)\backslash \{\theta\}$ is connected
and regular.
\end{lem}

\begin{proof}
Because $B/A$ is regular, each row and column of $P$ contains a
non-zero entry. Lemma~2.1 of~\cite{Eric} implies that if each row
and column does not contain more than one non-zero element, then
$B/A$ is nilpotent.

Suppose $B/A$ is not nilpotent. Then, without loss of generality,
we may assume that some row of $P$ contains more than one non-zero
element. Say $p_{j,i_1}, p_{j,i_2} \neq \theta$. If
$\Lambda = \{j\}$, because $B/A$ is regular, all
columns are non-zero, and hence all elements of $P$ are non-zero.
Otherwise, let $j' \in \Lambda, \, j' \neq j$.
Because $B/A$ is regular, there exists $\alpha \in
I$ such that $p_{j',\alpha} \neq \theta$. Now we
have
$$((p_{j,i_1}^{-1}; i_1, j'),(p_{j,i_2}^{-1};i_2,j')) =
(\lambda_1((p_{j,i_1}^{-1}; i_1, j'),(p_{j,i_2}^{-1};i_2,j'),
(p_{j',\alpha}^{-1}; \alpha, j)),$$ $$\rho_1((p_{j,i_1}^{-1}; i_1,
j'),(p_{j,i_2}^{-1};i_2,j'), (p_{j',\alpha}^{-1}; \alpha, j))).$$
Hence, because $S$ is \B, there is an edge in ${\mathcal N}_{S/A}$
between $(p_{j,i_1}^{-1}; i_1, j')$ and $(p_{j,i_2}^{-1};i_2,j')$.
Consequently,  by Lemma~\ref{finite-nilpotent}, there exists a
positive integer $m$ and elements $x, y, w_{1}, w_{2}, \ldots,
w_{m}$ of $\langle (p_{j,i_1}^{-1}; i_1, j'),
(p_{j,i_2}^{-1};i_2,$ $j') \rangle^{1}$ such that $x =
\lambda_{m}(x, y,$ $ w_{1}, w_{2}, \ldots, w_{m})$, $y =
\rho_{m}(x, y, w_{1}, w_{2}, \ldots, w_{m})$ and $x\neq y$.

If $p_{j',i_1}=\theta$ then, for $ g \in G$ and working in $S/A$,
$$\langle(p_{j,i_1}^{-1};
i_1, j'), (p_{j,i_2}^{-1};i_2,j') \rangle (g; i_1, j') =\theta
.$$ Hence $\{x, y, w_{1},$ $ w_{2},
\ldots, w_{m}\} \subseteq G_{i_2, j'}$ and
$p_{j',i_2} \neq \theta$, because it is impossible that $\theta
\in \{x, y\}\{w_{1}, w_{2}, \ldots, w_{m}\}$ or $\theta \in
\{w_{1}, w_{2}, \ldots, w_{m}\}\{x, y\}$.

Consequently, the group $G_{i_2, j'} \cong G$ is
not nilpotent. However, as a subsemigroup of $S$, it is \B\ and
thus, by Lemma~\ref{RN Groups}, it is a nilpotent group. This
yields a contradiction. So we have proved that $p_{j',i_1} \neq
\theta$. Similarly $p_{j',i_2} \neq \theta$. Thus, $p_{k,i_1} \neq
\theta, p_{k,i_2}\neq \theta$ for all $k \in
\Lambda$. That is, columns $i_1$ and $i_2$ of
$P$ do not contain $\theta$.

Let $i_3 \in I$. Because $B/A$ is regular, there
exists $l \in \Lambda$ such that $p_{l,i_3} \neq
\theta$. As also $p_{l,i_1} \neq \theta$, the above yields that
column $i_3$ of $P$ also not contain $\theta$. So all entries of
$P$ are different from $\theta$. This finishes the first part of
the result.

The second part easily can be verified.
\end{proof}

Next we describe how a completely $0$-simple factor fits into  a
\B\ semigroup, in particular, we investigate the restrictions on
how elements in different principal factors multiply. To do so we
introduce some notations. Let $S$ be a semigroup with a
zero $\theta $. For an ideal $J$ of $S$ let
$$F_J(S) = F_J = \{s \in S \backslash J\mid sj \neq \theta  \text{ or } js\neq
\theta  \text{ for some } j \in J\}$$ and let
$$F'_J(S) = F'_J =  S \backslash (F_J \cup (J \backslash \{\theta \})).$$
Obviously, $F'_J \cup J$ is an ideal of $S$ and if $x \in F'_J$
and $y \in J$, then $\langle x, y \rangle$ is nilpotent, i.e.
there is no edge in ${\mathcal N}_{S}$ between $x$ and $y$.

Note however that for an arbitrary   \B\ semigroup $S$ with zero the set $F'_J(S)$ is not necessarily an ideal for an ideal $J$  of $S$.
For example let $S=\{\theta ,a,b,c\}$ be the semigroup with Cayley table
$$
\begin{tabular}{ c|c c c c}
           & $\theta$ & $a$        & $b$      & $c$\\ \hline
$\theta$   & $\theta$ & $\theta$   & $\theta$ & $\theta$\\
$a$        & $\theta$ & $\theta$   & $a$      & $\theta$\\
$b$        & $\theta$ & $a$        & $b$      & $a$  \\
$c$        & $\theta$ & $\theta$   & $a$      & $\theta$\\
\end{tabular}
$$
Then $J=\{\theta,a\}$ ia an ideal of $S$, $F_J(S)=\{b\}$ and $F'_J(S)=\{\theta,c\}$. Since $bc=a$, we have that $F'_J(S)$
is not an ideal of $S$. Note furthermore that $S$ is commutative, thus it is $(MN)$ and hence \B.

Recall that if $B$ is a band, i.e. a semigroup of idempotents,
then a semigroup $S$ is said to be a $B$-band union of
subsemigroups $S_{b}$, with $b\in B$, provided that
$S=\bigcup_{b\in B} S_{b}$, a disjoint union and
$S_{b_{1}}S_{b_{2}}\subseteq S_{b_{1}b_{2}}$ for all
$b_{1},b_{2}\in B$.

\begin{lem}\label{complete}
Let $S$ be a \B\ finite semigroup. Assume $J$ is an ideal of $S$
and $J=\mathcal{M}^0(G,I,\Lambda;P)$, a regular Rees matrix
semigroup. Assume $J$ is not nilpotent. The following properties
hold.

\begin{enumerate}
    \item \label{complete1} $F_J = \{s \in S \backslash J\mid sj \neq \theta,  js\neq \theta, j'sj
    \neq \theta  \text{ for all } j, j' \in J \backslash \{\theta \}\}$ and hence
    $F_J \cup (J \backslash \{\theta \})$ is a subsemigroup of $S$ and $F'_J(S)$ is an ideal of $S$.
    \item \label{complete2} If $a \in F_J \cup (J \backslash \{\theta \})$, then there exists a unique pair
    $(e, f)= (e(a), f(a)) $ $ \in I \times \Lambda$
    such that there is no edge in
${\mathcal N}_S$ between $a$ and $(g; e, f)$ for every $g \in G$,
i.e. $\langle (g; e, f), a\rangle$ is nilpotent for all
    $g \in G$ and there is an edge in ${\mathcal
    N}_S$ between $a$ and each element in $J \backslash ( G_{e,f} \cup \{ \theta \})$ .
    Furthermore,
     $$\Phi: S\longrightarrow
    (I\times \Lambda )^{0}$$
    with $$ a\mapsto \left\{ \begin{array}{ll}
                               (e(a),f(a)) & \mbox{if} ~a\in F_{J}\cup (J\backslash \{ \theta \} )\\
                               \theta & \mbox{otherwise}
                            \end{array} \right.$$
    is a semigroup  epimorphism from $S$ to the rectangular band $(I \times \Lambda)^0$. We
    call $\Phi$ a \B\ homomorphism. Also we have $a(J
    \backslash \{\theta \}) \subseteq G_{e(a), \ast}$ and $(J
    \backslash \{\theta \}) a \subseteq G_{\ast, f(a)}$ for $a \in F_J \cup (J \backslash \{\theta \})$.
    \item \label{complete3} $S$ is the $\Phi (S)$-band union of the  semigroups
     $\Phi^{-1}(\Phi(a))$
     with $a \in S$, that is $S =
    \bigcup_{a \in S} \Phi^{-1}(\Phi(a))$ and $\Phi^{-1}(\Phi(a)) \;
    \Phi^{-1}(\Phi(b)) \subseteq \Phi^{-1}(\Phi(ab))$. The sets $\Phi^{-1}(\Phi(a))$ will be  called the \B\ classes of
    $S$.
    \item \label{complete4} If $a, b \in F_J \cup (J \backslash \{\theta \})$ and $\Phi(a) \neq \Phi(b)$,
    then there is an edge in ${\mathcal N}_{S}$ between
    $a$ and $b$.
    \item \label{complete5} $(S \backslash F_J) / J$ is an ideal of $S / J$ and $J$ is an ideal of $F_J
    \cup J$.
\end{enumerate}
\end{lem}

\begin{proof}
(1) Because $J$ is \B, Lemma~\ref{RN Rees semigroups} yields that
all entries of $P$ are non-zero. Assume $a \in F_J \cup (J
\backslash \{\theta \})$. Then there exists $(g;i,\lambda) \in J$
such that $a(g;i,\lambda) \neq \theta $ or $(g;i,\lambda)a \neq
\theta $. Since all entries of $P$ are non-zero, we get that $aj,
ja, jaj'$ are non-zero for all $j, j' \in J\backslash \{ \theta \}
$. Then $$F_J = \{s \in S \backslash J\mid sj \neq \theta ,\;
js\neq \theta ,\;  j'sj \neq \theta  \text{ for all } j, j' \in J
\backslash \{\theta \}\}.$$ Hence $F_J \cup (J \backslash \{\theta
\})$ is a subsemigroup of $S$.

Now let $a \in F'_J(S)$. Suppose that there exist $b \in S$ such that $ab \not \in F'_J(S)$. Thus $ab \in F_J \cup (J \backslash \{\theta\})$.
Since $F_J \cup (J \backslash \{\theta\})$ is a subsemigroup of $S$ and $J\backslash \{\theta\} \neq \emptyset$, there exists
$j \in J \backslash \{\theta\}$ such that $abj \neq \theta$. Therefore, $a\in F_J \cup (J \backslash \{\theta\})$ and we get a contradiction.
Hence $F'_J(S)$ is a right ideal of $S$. Similarly we can prove that $F'_J(S)$ is a left ideal. Thus $F'_J(S)$ is an ideal of $S$.

(2) Let $a \in F_J \cup (J \backslash \{\theta \})$. The group $G$
is isomorphic with a subsemigroup of $S$. Hence $G$ is \B\ too.
Hence, by Lemma~\ref{RN Groups}, $G$ is nilpotent. Because, by
assumption $J$ is not nilpotent and $G$ is nilpotent, we have
$\abs{I}> 1$ or $\abs{\Lambda}> 1$. We suppose that $\abs{I}> 1$.

Let $(g;i,j)\in J$ with $g\in G$. Since $J$ is an ideal of $S$,
then there exist $i', i'' \in I$, $\lambda', \lambda'' \in
\Lambda$ such that $a(g;i,\lambda) =(g';i',\lambda),
a^2(g;i,\lambda) =(g'';i'',\lambda), (g;i,\lambda)a
=(k;i,\lambda')$ and $(g;i,\lambda)a^2 =(k';i,\lambda'')$ for
some $g',g'',k, k' \in G$. Therefore, for every $1 \leq n$,
$$(\lambda_{2n}(a, a^2, (g;i,\lambda), 1, \ldots, 1),\rho_{2n}(a, a^2,
(g;i,\lambda), 1, \ldots, 1)) = $$ $$((m_{2n};i',\lambda'),
(m'_{2n};i'',\lambda'')),$$
$$(\lambda_{2n-1}(a, a^2, (g;i,\lambda), 1, \ldots,
1),\rho_{2n-1}(a, a^2, (g;i,\lambda), 1, \ldots, 1)) =$$
$$((m_{2n-1};i',\lambda''), (m'_{2n-1};i'',\lambda')),$$ for some
$m_i,m'_i \in G$. If $(i', \lambda') \neq (i'', \lambda'')$ (in
particular $a \neq a^2$) then
$$\lambda_{n'}(a, a^2, (g;i,\lambda), 1, \ldots, 1) \neq \rho_{n'}(a,
a^2, (g;i,\lambda), 1, \ldots, 1),$$ for every $n' \geq 1$. As $S$
is finite, there exist distinct positive integers $t$ and $r$ such
that
$$(\lambda_{t}(a, a^2, (g;i,\lambda), 1, \ldots, 1),\rho_{t}(a, a^2,
(g;i,\lambda), 1, \ldots, 1)) = $$
$$(\lambda_{r}(a, a^2, (g;i,\lambda), 1, \ldots, 1),\rho_{r}(a, a^2,
(g;i,\lambda), 1, \ldots, 1)).$$ Because  $S$ is \B, we thus
obtain an edge in ${\mathcal N}_S$ between $a$ and $a^2$ , a
contradiction. Therefore $i'=i''$ and $\lambda'=\lambda''$.

Let $r \in G$. Then there exist $w, z, r', r'' \in G$ such that
$$a (r; i',\lambda')= a (g';i',\lambda)(r';i,\lambda')=
aa(g;i,\lambda)(r';i,\lambda')=$$ $$
(g'';i',\lambda)(r';i,\lambda')=(w;i',\lambda'),$$
$$(r;i',\lambda')a= (r'';i',\lambda)(k;i,\lambda')a=
(r'';i',\lambda)(g;i,\lambda)aa =$$ $$
(r'';i',\lambda)(k';i,\lambda')=(z;i',\lambda').$$ Hence
$G_{i',\lambda'}$ is an ideal of the semigroup $\langle a \rangle
\cup G_{i',\lambda'}$. Since $G \cong G_{i',\lambda'}$ is
nilpotent, Lemma~\ref{empty ideal} yields that there is no edge in
${\mathcal N}_S$ between $a$ and any element of $G_{i',\lambda'}$.

Now let $(s;\alpha,\beta)\in J, (\alpha,\beta) \neq
(i',\lambda')$. We have, for every $1 \leq n$,
$$(\lambda_{2n}(a, (s;\alpha,\beta),(1;i',\lambda'), 1, \ldots, 1),
\rho_{2n}(a, (s;\alpha,\beta),(1;i',\lambda'), 1, \ldots, 1)) = $$
$$ ((v_{2n};i',\lambda'), (v'_{2n};\alpha, \beta)),$$
$$(\lambda_{2n-1}(a, (s;\alpha,\beta),(1;i',\lambda'), 1, \ldots,
1),\rho_{2n-1}(a, (s;\alpha,\beta),(1;i',\lambda'), 1, \ldots, 1))
=$$
$$((v_{2n-1};i',\beta), (v'_{2n-1};\alpha, \lambda')),$$ for some
$v_i, v'_i \in G$. Because  $(\alpha,\beta) \neq (i',\lambda')$,
we thus have that
$$\lambda_{n'}(a, (s;\alpha,\beta),(1;i',\lambda'), 1, \ldots, 1)
\neq \rho_{n'}(a, (s;\alpha,\beta),(1;i',\lambda'), 1, \ldots,
1),$$ for every $n' \geq 1$. In a similar way  as above for $a$
and $a^2$, there is an edge in ${\mathcal N}_S$ between
$(s;\alpha,\beta)$ and $a$ for every $s \in G$. Therefore, between
$a$ and all elements of $J \backslash (G_{i',\lambda'} \cup
\{\theta\})$ there are edges in ${\mathcal N}_S$ and
$(i',\lambda')$ is the unique element of $I \times \Lambda$ such
that there is no edge in ${\mathcal N}_S$ between $a$ and all
elements of $G_{i',\lambda'}$ .

So we have shown that if $a\in F_J$ and
$a(g;i,\lambda)=(g';i',\lambda), (g;i,\lambda)a=(k';i,\lambda')$
then $\langle a, (g;i',\lambda')\rangle$ is nilpotent for any $g
\in G$ and $\langle a, (g;\alpha,\beta)\rangle$ is not nilpotent
for any $g \in G$ and $(\alpha,\beta) \neq (i',\lambda')$. Since
$\theta \neq a(J\backslash \{\theta\})$ and $\theta\neq
(J\backslash \{\theta\})a$, it follows that $a(J\backslash
\{\theta\}) \subseteq G_{i',\ast}$ and $(J\backslash \{\theta\})a
\subseteq G_{\ast,\lambda'}$. This fact will be used twice in the
proof.

Note that if $(i, \lambda) \neq (i',\lambda')$ then it is readily
verified that $\langle (g;i,\lambda), (g';i',\lambda')\rangle$ is
not nilpotent. So there is an edge in ${\mathcal N}_{S}$ between
$(g;i,\lambda)$ and $(g';i',\lambda')$. On the other hand, each
$\langle (g;i,\lambda), (g';i,\lambda)\rangle$ is nilpotent (as a
subsemigroup of $G_{i,\lambda} \cong G$).

Because of the above, we
now can define a function
$$\Phi: S \rightarrow (I \times \Lambda)^0$$
as follows. For $a \in S \backslash F'_J, \Phi(a)= (i',\lambda')$
where $(i', \lambda')\in I \times \Lambda$ is such that
$a(J\backslash \{\theta\}) \subseteq G_{i',\ast}$ and
$(J\backslash \{\theta\})a \subseteq G_{\ast,\lambda'}$. If $a \in
F'_J$ then we define $\Phi(a)=\theta$.

Consider $I \times \Lambda$ as a rectangular band (for the natural
multiplication). Hence $(I \times \Lambda)^0$ is a band with a
zero $\theta$.

We claim that $\Phi$ is a semigroup homomorphism.

So let $x, y \in S$. If $x \in F'_J$ or $y \in F'_J$, then $xy \in
F'_J$ (as by part (1), $F'_J$ is an ideal of $S$); hence $\Phi(xy)=
\Phi(x)\Phi(y)=\theta$.

Assume now that $x,y \in S \backslash F'_J$. Then there exist
unique $(i_1, \lambda_1),(i_2, \lambda_2) \in I \times
\Lambda$, such that $x(J\backslash \{\theta\}) \subseteq
G_{i_1,\ast}, (J\backslash \{\theta\})x \subseteq
G_{\ast,\lambda_1}, y(J\backslash \{\theta\}) \subseteq
G_{i_2,\ast}$ and $(J\backslash \{\theta\})y \subseteq
G_{\ast,\lambda_2}$. Consequently, $xy(J\backslash \{\theta\})
\subseteq xG_{i_2,\ast} \subseteq G_{i_1,\ast}$ and $(J\backslash
\{\theta\})xy \subseteq G_{\ast,\lambda_1}y \subseteq G_{\ast,
\lambda_2}$. Hence,  by the fact mentioned above, $\Phi(xy) =
(i_1, \lambda_2)= (i_1,\lambda_1)(i_2,\lambda_2) =
\Phi(x)\Phi(y)$. So, indeed, $\Phi$ is a semigroup homomorphism.

(3) Because each element of $(I \times \Lambda)^0$ is idempotent,
one has that $\Phi^{-1}(\Phi(a))$ is a subsemigroup of $S$ for
each $a \in S$. As $\Phi$ is surjective, we get that $S =
\bigcup_{x \in \Phi(S)} \Phi^{-1}(x)$, a disjoint union and
$\Phi^{-1}(x) \Phi^{-1}(y) \subseteq \Phi^{-1}(xy).$ Hence
statement (3) follows.

(4) Assume $a, b \in F_J \cup (J \backslash \{\theta\})$ are such
that $(i, \lambda) = \Phi(a) \neq \Phi(b) = (i',\lambda')$. We
have, for every $ 1 \leq n$,
$$(\lambda_{2n}(a, b,(1;i',\lambda'), 1, \ldots, 1),
\rho_{2n}(a, b,(1;i',\lambda'), 1, \ldots, 1)) $$ $$=
((m_{2n};i,\lambda),(m'_{2n};i',\lambda'),$$
$$(\lambda_{2n-1}(a, b,(1;i',\lambda'), 1, \ldots,
1),\rho_{2n-1}(a, b,(1;i',\lambda'), 1, \ldots, 1)) =$$
$$((m_{2n-1};i,\lambda'), (m'_{2n-1};i',\lambda)),$$
for some $m_k, m'_k \in G$. Because $(i, \lambda) \neq (i',
\lambda')$,
$$\lambda_{n'}(a, b,(1;i',\lambda'), 1, \ldots, 1) \neq
\rho_{n'}(a, b,(1;i',\lambda'), 1, \ldots, 1),$$ for every $n'
\geq 1$. Since $S$ is finite, there exist distinct positive
integers $t$ and $r$ such that
$$(\lambda_{t}(a, b,(1;i',\lambda'), 1, \ldots, 1),
\rho_{t}(a, b,(1;i',\lambda'), 1, \ldots, 1)) = $$
$$(\lambda_{r}(a, b,(1;i',\lambda'), 1, \ldots, 1),
\rho_{r}(a, b,(1;i',\lambda'), 1, \ldots, 1)).$$ As $S$ is \B, we
obtain that there is an edge in ${\mathcal N}_S$ between $a$ and
$b$.

(5) By part (1), $F'_J$ is an ideal of $S$. Now since $J$ is an ideal of $S$, statement (5) is
obvious.
\end{proof}

Every finite semigroup $S$ has principal series:
$$S= S_1 \supset S_2 \supset \cdots \supset S_m \supset S_{m+1} = \emptyset .$$ That is, each
$S_i$ is an ideal of $S$ and there is no ideal of $S$ strictly
between $S_i$ and $S_{i+1}$ (for convenience we call the empty set
an ideal of $S$). Each principal factor $S_i / S_{i+1} (1 \leq i
\leq m)$ of $S$ either is completely $0$-simple, completely simple
or null. Every completely $0$-simple factor semigroup is
isomorphic with a regular Rees matrix semigroup over a finite
group $G$.

Suppose $S$ is \B. Then, by Lemma~\ref{RN Rees semigroups}, every
principal factor which is a regular Rees matrix semigroup is
nilpotent or all entries of the respective sandwich matrix are
non-zero, that is, it is a union of groups. If $S_m$ and all $S_i
/ S_{i+1}$ are nilpotent, then by Lemma~\ref{empty}, $S$ is
nilpotent as well.

\begin{defn}
If $S_i / S_{i+1}$ is not nilpotent (thus $S_i \backslash S_{i+1}$
is a semigroup) and there is no edge in ${\mathcal N}_{S}$ between
any element of $S_i \backslash S_{i+1}$ and any element of
$S_{i+1}$ then we say that $S_i \backslash S_{i+1}$ is a root of
$S$.
\end{defn}

In  case $S_i \backslash S_{i+1}$ is a root of $S$, then it
follows from Lemma~\ref{complete}.(\ref{complete2}) that if $i' >
i$ and $S_{i} \backslash S_{i+1} \subseteq F_{S_{i'} / S_{i'+1}}(S
/ S_{i'+1})$ then $S_{i'} / S_{i'+1}$ is nilpotent.

Note that there exist \B\ finite semigroups with more than one
root. An example is the semigroup $T=\{a, b, c, d, e\}$ with
multiplication table $$
\begin{tabular}{ c|c c c c c}
      & $a$ & $b$ & $c$ & $d$ & $e$\\ \hline
  $a$ & $a$ & $a$ & $a$ & $a$ & $a$\\
  $b$ & $a$ & $b$ & $a$ & $a$ & $b$\\
  $c$ & $a$ & $a$ & $c$ & $d$ & $a$\\
  $d$ & $a$ & $a$ & $c$ & $d$ & $a$\\
  $e$ & $a$ & $e$ & $a$ & $a$ & $e$\\
\end{tabular}$$
The subsemigroups $\{b, e\}$ and $\{c, d\}$ are roots of $T$.

It is now convenient to identify the non-zero elements of $S / I$
with those of $S \backslash I$ for $I$ an ideal of $S$. Let $j < i\leq m$ and let $B = F'_{S_i /
S_{i+1}}(S/S_{i+1}) \cap (S_j \backslash
S_{j+1})$. We claim that $B \cup
S_{j+1}$ is an ideal of $S$.

Indeed, let $x\in B \cup S_{j+1}$ and let $a \in S$.
We need to show that $ax \in B\cup S_{j+1}$ and $xa\in B\cup S_{j+1}$. We only prove the former, the other one is shown similar. We give a proof by contradiction. So suppose
$ax \not\in B\cup S_{j+1}$.
 Since $x \in S_j$ and $S_j$ is an ideal of $S$, we have $ax \in S_j \backslash S_{j+1}$.
 Furthermore, $ax \not \in B$ and $ax\not\in S_{i}\backslash S_{i+1}$ imply  $ax \in F_{S_i /
S_{i+1}}(S/S_{i+1})$ and thus there exists an element $y$ in $S_i \backslash S_{i+1}$ such that $axy \not \in S_{i+1}$ or $yax \not \in S_{i+1}$.
Clearly  $ya \in S_i$. Now since $x$ in $B$, $x \in F'_{S_i /S_{i+1}}(S/S_{i+1})$ and thus $xy$ and $yax$ are in $S_{i+1}$. However, this is in
contradiction with $axy \not \in S_{i+1}$ or $yax \not \in S_{i+1}$. This proves that indeed $ax\in B\cup S_{j+1}$.

Now, since $S_{j+1} \subseteq B \cup S_{j+1} \subseteq S_j$ and because there is no ideal strictly between $S_{j+1}$ and $S_{j}$, we get $B= \emptyset$ or $B= S_j \backslash S_{j+1}$. It therefore follows that
if $B \neq \emptyset$ then $S_j \backslash S_{j+1} \subseteq
F'_{S_i / S_{i+1}}(S / S_{i+1})$. On the other hand if $B =
\emptyset$ then $S_j \backslash S_{j+1} \subseteq F_{S_i /
S_{i+1}}(S / S_{i+1})$.

For $1\leq i \leq m$ let
$$ i^{*}=\{ j\leq i \mid S_{j}\backslash S_{j+1} \subseteq F_{S_{i}/S_{i+1}}(S/S_{i+1}) \cup S_{i}\backslash S_{i+1}\}$$
and
$$S^{(i)}  = \bigcup_{ j \in i^{\star}} (S_j \backslash S_{j+1}).$$
Since $F_{S_{i}/S_{i+1}}(S/S_{i+1}) \subseteq S\backslash S_{i}$, $S\backslash S_{i} = \bigcup_{1\leq k < i} (S_k \backslash S_{k+1})$ and the fact mentioned above we have
$$S^{(i)} =F_{S_i / S_{i+1}}(S / S_{i+1}) \cup S_i \backslash S_{i+1}.$$

\begin{defn}
If $S_i \backslash S_{i+1}$ is a root then  the set $S^{(i)}$ is
called the stem of $S_{i}\backslash S_{i+1}$. In this case,
because of Lemma~\ref{complete}, if $S_{j}\backslash
S_{j+1}\subseteq S^{(i)}$ then there is a path in
$\mathcal{N}_{S}$ between any two elements $s,t\in S^{(i)}$.
\end{defn}

Note that for every set $T = S_i \backslash S_{i+1}$ we have
three possible
cases: (i) $T$ is a root, (ii)  $T$ is not a root
and there exists a non-nilpotent  semigroup $S_j \backslash S_{j+1}$
such that $T \subseteq F_{S_j / S_{j+1}}(S / S_{j+1})$ (so in this
case there is an edge in $\mathcal{N}_{S}$ between some
element of $T$ and some element of $S_{j}\backslash S_{j+1}$),
(iii) $T$ is not a root and if $T \subseteq F_{S_j / S_{j+1}}(S /
S_{j+1})$ then $S_j / S_{j+1}$ is nilpotent. The third case will
be given a special name.

\begin{defn} We say that $S_{i}\backslash S_{i+1}$ is an isolated
subset provided that $S_i \backslash S_{i+1}$ is not a root that
satisfies the property that if $S_i \backslash S_{i+1} \subseteq
F_{S_j / S_{j+1}}(S / S_{j+1})$ then $S_j / S_{j+1}$ is nilpotent.
\end{defn}

\begin{defn} \label{def-connection}
Suppose $S^{(i_1)}$ and $S^{(i_2)}$ are two distinct stems of $S$.
If $S^{(i_1)} \cap S^{(i_2)} \neq \emptyset$ then $ S_j \backslash
S_{j+1}\subseteq S^{(i_1)} \cap S^{(i_2)}$ for some $j$. We call
$S_j \backslash S_{j+1}$ a connection between the stems
$S^{(i_1)}$ and $S^{(i_2)}$.
\end{defn}
The reason for this name is clear as in the upper non-nilpotent
graph ${\mathcal N}_S$ there is a path from any element of
$S^{(i_1)}$ to any element of $S^{(i_2)}$ via a vertex in $S_j
\backslash S_{j+1}$. For example, in the semigroup $T=\{\theta,
a,b,c,d,f \}$ with multiplication table
$$
\begin{tabular}{ c|c c c c c c}
           & $\theta$ & $a$      & $b$      & $c$      & $d$      & $f$\\ \hline
  $\theta$ & $\theta$ & $\theta$ & $\theta$ & $\theta$ & $\theta$ & $\theta$\\
  $a$      & $\theta$ & $a$      & $b$      & $\theta$ & $\theta$ & $b$\\
  $b$      & $\theta$ & $a$      & $b$      & $\theta$ & $\theta$ & $b$\\
  $c$      & $\theta$ & $\theta$ & $\theta$ & $c$      & $d$      & $d$\\
  $d$      & $\theta$ & $\theta$ & $\theta$ & $c$      & $d$      & $d$\\
  $f$      & $\theta$ & $a$      & $b$      & $c$      & $d$      & $f$\\
\end{tabular}$$
We prove that $T$ is \B\ by contradiction. So, suppose the contrary. Hence,
there exist  distinct elements $x,y\in T$, elements $w_{1},\ldots , w_{m}\in T^{1}$, an ideal $I$ of $\langle x,y,w_{1},\ldots ,
w_{m}\rangle$
such that $$\lambda_{t}(x, y, w_{1},
w_{2}, \ldots, w_{t}) = \lambda_{m}(x, y, w_{1},
w_{2}, \ldots, w_{m}),$$ $$\rho_{t}(x, y, w_{1}, w_{2}, \ldots,
w_{t}) = \rho_{m}(x, y, w_{1}, w_{2}, \ldots,
w_{m}),$$ $\lambda_m \neq \rho_m$, $\lambda_m, \rho_m\not\in I$  and  $\langle \lambda_i, \rho_i \rangle$ is nilpotent in $\langle x,y,w_{1},\ldots , w_{m}\rangle/I$ for some $0 \leq i \leq m$. Because the subsemigroups $\{\theta, a, b\}$ and
$\{\theta, c,d\}$ are ideals of $S$ and $\{\theta, a, b\}\{\theta, c,d\}=\{\theta, c,d\}\{\theta, a, b\}=\{\theta\}$, the equalities imply that neither $\lambda_i \in \{\theta, a, b\},\,  \rho_i \in \{\theta, c,d\}$ nor $\lambda_i \in \{\theta, c, d\},\,  \rho_i \in \{\theta, a,b\}$. Since $\langle \lambda_i, \rho_i \rangle$ is nilpotent in $\langle x,y,w_{1},\ldots , w_{m}\rangle/ I$, either $\{\lambda_i, \rho_i\}=\{b,f\}$ or $\{\lambda_i, \rho_i\}=\{d,f\}$.
Suppose that $\{\lambda_i, \rho_i\}=\{b,f\}$. Since $\theta \not \in \{\lambda_{i+1}, \rho_{i+1}\}$, $w_{i+1} \in \{a,b,f,1\}$. Then $\lambda_{i+1}= \rho_{i+1}=b$, a contradiction. Similarly one obtains  a contradiction for $\{\lambda_i, \rho_i\}=\{d,f\}$. So, indeed, $T$ is \B.

Note that the set $\{f\}$ is a connection between the stems $\{a,b,f\}$ and $\{c,d,f\}$.

Let $S$ be a \B\ finite semigroup with principal series
$$S= S_1 \supset S_2 \supset \ldots \supset S_m \supset S_{m+1} = \emptyset
.$$ If $S_{i}/S_{i+1}$ is not nilpotent then
$(S_{i}/S_{i+1})^{0}=\mathcal{M}^{0}(G,I,\Lambda;P)$ (with all
entries of $P$ non-zero) and we denote by $\Phi_{i}:
(S/S_{i+1})^{0} \rightarrow (I\times \Lambda )^{0}$ the \B\
homomorphism obtained in Lemma~\ref{complete}.

\begin{thm}\label{principal series}
Let $S$ be a \B\ finite semigroup with principal series
$$S= S_1 \supset S_2 \supset \ldots \supset S_m \supset S_{m+1} = \emptyset.$$ The following properties hold.

\begin{enumerate}
    \item \label{ps1} The union $K$ of all isolated subsets $S_i
    \backslash S_{i+1}$ is  the largest nilpotent ideal of $S$  and it is the set of all isolated vertices in
    ${\mathcal N}_{S}$ (possibly $K=\emptyset$).
    \item \label{ps2} If $S_{i_1} \backslash S_{i_1+1} \subseteq S^{(i_2)}$ and $S_{i_2} \backslash S_{i_2+1} \subseteq
    S^{(i_3)}$, then $S_{i_1} \backslash S_{i_1+1} \subseteq
    S^{(i_3)}$. If $S_{i_2} \backslash S_{i_2+1} = \mathcal{M}(G_2,I_2,\Lambda_2;P_2)$ and $S_{i_3}
    \backslash S_{i_3+1} = \mathcal{M}(G_3,I_3,\Lambda_3;P_3)$ are not nilpotent
    and $\Phi_{2}: (S / S_{i_2+1})^{0} \longrightarrow
    (I_2\times \Lambda_2)^{0}$, $\Phi_{3}: (S / S_{i_3+1})^{0} \longrightarrow
    (I_3\times \Lambda_3)^{0}$ are \B\ homomorphisms, then $\Phi_{3}(a)=\Phi_{3}((g;$ $\Phi_{2}(a)))$, for every $a \in
    S_{i_1} \backslash S_{i_1+1}$ and $g\in G_2$.
    \item \label{ps3} If $S_i \backslash S_{i+1}$ is not an isolated subset and not
    a root, then there exists a root $S_j \backslash S_{j+1}$
    such that $i < j$ and $S_i \backslash S_{i+1} \subseteq S^{(j)}$.
    \item \label{ps4} $S \backslash K = \bigcup S^{(i)}$, where the union
    runs over all $i$ with $S_i \backslash S_{i+1}$ a root.
    \item \label{ps5} Every stem $S^{(i)}$ is a subsemigroup.
    \item \label{ps6} $(S_i \backslash S_{i+1})(S_j \backslash S_{j+1}) \subseteq
    K$ if and only if $S_i \backslash S_{i+1}$ and $S_j \backslash S_{j+1}$
    are not in a same stem.
\end{enumerate}
\end{thm}

\begin{proof}
(1) First suppose that $S$ does not have any isolated subset. Then
for every principal factor $S_i / S_{i+1}$, the subset $S_i
\backslash S_{i+1}$ is a root or there exists non-nilpotent
semigroup $S_j \backslash S_{j+1}$ such that $S_i \backslash
S_{i+1} \subseteq F_{S_j / S_{j+1}}(S / S_{j+1})$. Because of
Lemma~\ref{complete}.(\ref{complete2}), in both cases, all
elements of $S_i \backslash S_{i+1}$ are non-isolated vertices in
${\mathcal N}_{S}$. Therefore ${\mathcal N}_{S}$ has no isolated
vertex. Now suppose that $I$ is a nilpotent ideal of $S$. Then
${\mathcal N}_{I}$ is empty and by Lemma~\ref{empty ideal} the
elements of $I$ are isolated vertices in ${\mathcal N}_{S}$.
Hence, by the above, $I$ is empty.

Now assume $S$ has an isolated subset, i.e. we suppose that $K
\neq \emptyset$. Suppose that $a \in K, b \in S$ and $ab \notin
K$. Let $i$ and $j$ be such that $a \in S_i\backslash S_{i+1}$ and
$ab \in S_j\backslash S_{j+1}$. Because each $S_k$ is an ideal of
$S$, it is clear that $i < j$. Since $ab \notin K$ we have that
$S_j\backslash S_{j+1}$ is not an isolated subset. So either (i)
$S_j\backslash S_{j+1}$ is a root, or (ii) $S_j\backslash S_{j+1}$
is not a root and $S_j\backslash S_{j+1} \subseteq F_{S_k/
S_{k+1}}(S/ S_{k+1})$ for some $k$ with $j < k$  and $S_k/
S_{k+1}$ is not nilpotent.

We first show that case (i) leads to a
contradiction. So assume $S_j \backslash S_{j+1}$ is a root. In
particular, $S_j/ S_{j+1}$ is not nilpotent and
$\abs{S_j\backslash S_{j+1}} > 1$. Thus there exists $x \in
S_j\backslash S_{j+1}$ such that $ x \neq ab$. Because $a\in
S_{i}\backslash S_{i+1}$ and since $S_{i}\backslash S_{i+1}$ is an isolated
subset we get that $a \in F'_{S_j/ S_{j+1}}(S/ S_{j+1})$. Hence
$xa \in S_{j+1}$ and thus also $xab \in S_{j+1}$. On the other
hand, by Lemma~\ref{RN Rees semigroups}, $S_j\backslash S_{j+1}$
is a semigroup. But then $x, ab \in S_j \backslash S_{j+1}$ implies that $xab \in
S_j \backslash S_{j+1}$. This contradicts with $xab\in S_{j+1}$
and $S_{j}\backslash S_{j+1}$  being a root.

Next we show that case (ii) also leads to a contradiction. So
suppose that $S_j \backslash S_{j+1}$ is not a root and that there
exists a positive integer $k$ such that $ab \in F_{S_k/
S_{k+1}}(S/ S_{k+1})$, $j <  k$  and $S_k/ S_{k+1}$ is not
nilpotent. Choose $y \in S_k \backslash S_{k+1}$. Since
$S_{i}\backslash S_{i+1}$ is an isolated subset and because $a\in
S_{i}\backslash S_{i+1}$ we get that $a\in  F'_{S_k/ S_{k+1}}(S/
S_{k+1})$. Hence, $ya\in S_{k+1}$. As $S_{k+1}$ is an ideal of
$S$, we thus obtain that $yab\in S_{k+1}$. So, by
Lemma~\ref{complete}.(\ref{complete1}), $ab \not\in F_{S_k/
S_{k+1}}(S/ S_{k+1})$, a contradiction.

We thus have shown that indeed $K$ is a right ideal of $S$.
Similarly one shows that it is a left ideal and thus it is an
ideal.

We now prove that all elements of $K$ are isolated vertices. So
suppose the contrary and let $i$ be the largest positive integer
such that $S_i \backslash S_{i+1} \subseteq K$ and $S_i \backslash
S_{i+1}$ contains a non-isolated vertex, say $v$. Then there
exists an element $w \in S$ such that $\langle v, w \rangle$
is not nilpotent. Lemma~\ref{finite-nilpotent} implies that there
exists a positive integer $m'$, distinct elements $x, y\in \langle
v, w \rangle$ and elements $w_{1}, w_{2}, \ldots, w_{m'}\in
\langle v, w \rangle^{1}$, such that $x = \lambda_{m'}(x, y,
w_{1}, w_{2}, \ldots, w_{m'})$, $y = \rho_{m'}(x, y, w_{1}, w_{2},
\ldots, w_{m'})$. As $S$ is \B, we get that $\langle x, y\rangle$
is not nilpotent. Hence, since $\langle w \rangle$ is nilpotent,
we get that $\{x, y\} \not \subseteq \langle w \rangle$. As $S_i$
is an ideal of $S$ and $v \in S_i$, we clearly have $\langle w,
v\rangle \backslash \langle w \rangle \subseteq S_i$. Therefore
$\{x, y\} \cap S_i \neq \emptyset$. Again because $S_i$ is an
ideal of $S$ and since $m'\geq 1$, the
equalities
$$x = \lambda_{m'}(x, y,w_1,w_2, \ldots,w_{m'}),\; y =\rho_{m'}(x, y,w_1,w_2,
\ldots,w_{m'}),$$  imply that $x, y \in S_i$. Because $K$ is an
ideal we obtain in a similar manner that $x, y \subseteq K$. By
the maximality choice of $i$ we have that $x,y \in S_i\backslash
S_{i+1}$. Since $S$ is \B, the above equalities yield that
there is an edge between $x$ and $y$ in ${\mathcal N}_{S/S_{i+1}}$. So, by
Lemma~\ref{RN Rees semigroups}, $S_{i}\backslash S_{i+1}$ is a
non-nilpotent semigroup.

Since $S_i\backslash S_{i+1}$ is an isolated subset and
$S_i\backslash S_{i+1}$ is not nilpotent, it follows from the
definition of root, that there exist $i' > i, a \in S_i\backslash
S_{i+1}$ and $b \in S_{i'}\backslash S_{i'+1}$ such that there is
an edge in ${\mathcal N}_S$ between $a$ and $b$. Again with a
similar proof as above, there exist elements $a', b'$ in $\langle
a, b \rangle \cap S_{i'}$ such that $\langle a', b'\rangle$ is not
nilpotent and $a', b' \in K$. Let $i''>i$ be such that  $a' \in
S_{i''}\backslash S_{i''+1}$. As $K$ is a union of isolated
subsets, it follows that $S_{i''} \backslash S_{i''+1} \subseteq
K$ and $S_{i''}\backslash S_{i''+1}$ contains a non-isolated
vertex. This contradicts with the maximality of $i$. Hence we have
shown that indeed all elements of $K$ are isolated vertices.

We now show that if $S_i \backslash S_{i+1}$ is not an isolated
subset then it does not contain any isolated vertex; and hence $K$
is indeed the set of all isolated vertices. So suppose $S_i
\backslash S_{i+1}$ is not an isolated subset. Then either it is a
root or $S_i\backslash S_{i+1} \subseteq F_{S_j/ S_{j+1}}(S/
S_{j+1})$ for some $j > i$ with $S_j/ S_{j+1}$ not nilpotent. In
the former case, Lemma~\ref{RN Rees semigroups} yields that the upper non-nilpotent graph of
$S_{i}\backslash S_{i+1}$ is non-empty, connected and regular.
Hence $S_{i} \backslash S_{i+1}$ does not have any isolated
vertices in ${\mathcal N}_S$. In the second case, again by
Lemma~\ref{RN Rees semigroups}, there exists a $j > i$ such that
$S_j \backslash S_{j+1} = \mathcal{M}(G, I, \Lambda; P)$, $G$ a
nilpotent group, $\abs{I}+\abs{\Lambda} \geq 3$, all entries of
$P$ are non-zero and $S_i \backslash S_{i+1} \subseteq F_{S_j /
S_{j+1}}(S / S_{j+1})$. Again ${\mathcal N}_{S_{j} \backslash
S_{j+1}}$ is a non-empty connected and regular graph. By
Lemma~\ref{complete}, there exists a \B\ homomorphism $\Phi$ from
$(S /S_{j+1})^{0}$ to the rectangular band $(I \times
\Lambda)^{0}$. Furthermore, there is an edge in ${\mathcal N}_S$
between $a \in S_{i} \backslash S_{i+1}$ and any element in
$(S_{j} \backslash S_{j+1})\backslash G_{\Phi(a)}$. Hence, $S_{i}
\backslash S_{i+1}$ does not have any isolated vertices. So,
indeed, $K$ is the set of all isolated vertices.

As ${\mathcal N}_K$ is empty and $K$ is \B, by Lemma~\ref{empty graph}, the semigroup $K$ is
nilpotent. It remains to show that $K$ is the largest nilpotent
ideal of $S$.  To do so, let $a\in K'\backslash K$ with $K'$ an
ideal that is nilpotent. 
Then there exists $b\in S$ with $\langle a,b\rangle$ not
nilpotent. Since $S$ is \B\ and $\langle a,b\rangle \subseteq
\langle b \rangle \cup K'$, it follows with an argument as above
that there exist $e,f\in K'$ with $\langle e,f\rangle$ not
nilpotent. However, this contradicts with $K'$ being nilpotent.
So, indeed $K$ is the largest nilpotent ideal of $S$.

(2) If $i_1 = i_2$ or $i_2 = i_3$, then the statement is obvious.
Assume $i_1 \neq i_2$ and $i_2 \neq i_3$. Then the sets
$F_{S_{i_2} / S_{i_2+1}}(S / S_{i_2+1}) \cap (S_{i_1} \backslash
S_{i_1+1})$ and $F_{S_{i_3} / S_{i_3+1}}(S / S_{i_3+1})$ $\cap
(S_{i_2} \backslash S_{i_2+1})$ are not empty. Hence we get that
$S_{i_1} \backslash S_{i_1+1} \subseteq F_{S_{i_2} / S_{i_2+1}}(S
/ S_{i_2+1})$ and $S_{i_2} \backslash S_{i_2+1} \subseteq
F_{S_{i_3} / S_{i_3+1}}(S / S_{i_3+1})$. It follows that for $a\in
S_{i_{1}}\backslash S_{i_{1}+1}$ there exists $b\in
S_{i_{2}}\backslash S_{i_{2}+1}$ such that $ab \in S_{i_2} \backslash S_{i_2+1}$ or $ba \in S_{i_2} \backslash S_{i_2+1}$. Suppose that $ab \in S_{i_2} \backslash S_{i_2+1}$. Then $ab \in F_{S_{i_3} / S_{i_3+1}}(S / S_{i_3+1})$ and thus there exist $c\in
S_{i_{3}}\backslash S_{i_{3}+1}$ such that $abc \in S_{i_3} \backslash S_{i_3+1}$ or $cab \in S_{i_3} \backslash S_{i_3+1}$. If $abc \in S_{i_3} \backslash S_{i_3+1}$, then $bc \in S_{i_3} \backslash S_{i_3+1}$ and thus $a \in F_{S_{i_3} /S_{i_3+1}}(S / S_{i_3+1})$. Also if $cab \in S_{i_3} \backslash S_{i_3+1}$, then $ca \in S_{i_3} \backslash S_{i_3+1}$ and thus $a \in F_{S_{i_3} /S_{i_3+1}}(S / S_{i_3+1})$. Similarly if $ba \in S_{i_2} \backslash S_{i_2+1}$, then $a \in F_{S_{i_3} /
S_{i_3+1}}(S / S_{i_3+1})$. Consequently $S_{i_1} \backslash
S_{i_1+1} \subseteq S^{(i_3)}$.

Now assume $S_{i_2} \backslash S_{i_2+1} =
\mathcal{M}(G_2,I_2,\Lambda_2;P_2)$ and $S_{i_3} \backslash
S_{i_3+1} = \mathcal{M}(G_3,I_3,\Lambda_3;$ $P_3)$ are semigroups
that are not nilpotent. By Lemma~\ref{RN Rees semigroups},
$\abs{I_2}+\abs{\Lambda_2}>2, \abs{I_3}+\abs{\Lambda_3}>2$, $G_2$
and $G_3$ are nilpotent groups and all entries of both $P_2$ and
$P_3$ are non-zero. Because of
Lemma~\ref{complete}.(\ref{complete2}) there exist the \B\
homomorphisms $\Phi_{2}: (S/S_{i_{2}+1})^{0} \rightarrow
(I_{2}\times \Lambda_{2})^{0}$ and $\Phi_{3}: (S/S_{i_{3}+1})^{0}
\rightarrow (I_{3}\times \Lambda_{3})^{0}$.

For every $a \in S_{i_1} \backslash S_{i_1+1}$ and $g\in G_2, k,
l\in G_3$, by Lemma~\ref{complete}.(\ref{complete2}), there exist
elements $k',k'',l',l''\in G_3, g', g'' \in G_2$ such that
\begin{eqnarray*}
(k;\Phi_3((g;\Phi_2(a))))  &=& (g;\Phi_2(a)) (k';\Phi_3((g;\Phi_2(a)))),\\
a (g;\Phi_2(a)) &=& (g';\Phi_2(a)),\\
(g';\Phi_2(a)) (k';\Phi_3((g';\Phi_2(a)))) &=& (k'';\Phi_3((g';\Phi_2(a)))),\\
(l;\Phi_3((g;\Phi_2(a))))  &=&  (l';\Phi_3((g;\Phi_2(a)))) (g;\Phi_2(a)),\\
(g;\Phi_2(a)) a &=& (g'';\Phi_2(a)),\\
(l';\Phi_3((g'';\Phi_2(a)))) (g'';\Phi_2(a)) &=&
(l'';\Phi_3((g'';\Phi_2(a)))).
\end{eqnarray*}
As $G_2$ is nilpotent, there is no edge in ${\mathcal N}_S$
between $(g;\Phi_2(a))$, $(g';\Phi_2(a))$ and $(g'';\Phi_2(a))$.
Lemma~\ref{complete}.(\ref{complete4}) implies that
$\Phi_3((g;\Phi_2(a)))= \Phi_3((g';\Phi_2(a))) =
\Phi_3((g'';\Phi_2(a)))$. Therefore we have
\begin{eqnarray*}
a (k;\Phi_3((g;\Phi_2(a)))) &=& a (g;\Phi_2(a))
(k';\Phi_3((g;\Phi_2(a))))\\
 &=& (g';\Phi_2(a))
(k';\Phi_3((g;\Phi_2(a)))) \\
&=& (k'';\Phi_3((g;\Phi_2(a)))),\\
(l;\Phi_3((g;\Phi_2(a)))) a &=&
(l';\Phi_3((g;\Phi_2(a)))) (g;\Phi_2(a)) a\\
 &=&
(l';\Phi_3((g;\Phi_2(a)))) (g'';\Phi_2(a))\\
 &=& (l'';\Phi_3((g;\Phi_2(a)))).
\end{eqnarray*}
It implies that $\Phi_3(a)=\Phi_3((g;\Phi_2(a)))$.

(3) Since $S$ is finite and because  $S_i \backslash S_{i+1}$ is not an
isolated subset and not a root, there exists a positive integer
$i'$ such that $S_i \backslash S_{i+1} \subseteq F_{S_{i'}
/ S_{i'+1}}(S / S_{i'+1})$,  $S_{i'} / S_{i'+1}$ is not nilpotent
and if  $S_i \backslash S_{i+1} \subseteq
F_{S_{i''} / S_{i''+1}}(S / S_{i''+1})$ for some $i'' > i'$, then $S_{i''} /
S_{i''+1}$ is nilpotent. Clearly $S_i \backslash S_{i+1}
\subseteq S^{(i')}$.

If $S_{i'} \backslash S_{i'+1}$ is a root, then the statement
obviously holds. Otherwise, as $S_{i'} \backslash S_{i'+1}$ is not
nilpotent (and thus its vertices in ${\mathcal N}_S$ are not
isolated), we obtain from part (1) that $S_{i'}\backslash
S_{i'+1}$ is not an isolated subset. Hence there exists a positive
integer $j > i'$ such that $S_{i'} \backslash S_{i'+1} \subseteq
S^{(j)}$ and $S_{j} \backslash S_{j+1}$ is not nilpotent. By part
(2), $S_i \backslash S_{i+1} \subseteq S^{(j)}$. However, this
contradicts with the condition on $i'$.

(4) Let $s \in S \backslash K$ and let $i$ be such that $s\in S_i
\backslash S_{i+1}$. In particular $S_i \backslash S_{i+1}$ is not
an isolated subset. If $S_i \backslash S_{i+1}$ is a root then $s
\in S^{(i)}$. If $S_i \backslash S_{i+1}$ is not a root, then, by
part (3), $s \in S^{(k)}$ for some $k > i$ and $S_k \backslash
S_{k+1}$ is a root. It easily can be verified that any element of
a non-isolated subset is not an isolated vertex in ${\mathcal
N}_S$. Hence the statement follows.

(5) Suppose $S^{(i)}$ is a stem. As $S_i \backslash S_{i+1}$ is a
root, it is not nilpotent and by Lemma~\ref{RN Rees semigroups},
$S_i \backslash S_{i+1} = \mathcal{M}(G, I, \Lambda; P)$ a
regular Rees matrix semigroup with $\abs{I}+\abs{\Lambda}>2$,
$G$ a nilpotent group and all entries of $P$ non-zero and there
exists a \B\ homomorphism $\Phi: (S/S_{i+1})^{0} \rightarrow
(I\times \Lambda)^0$. Now suppose that $a,b \in S^{(i)}$. Hence
$\Phi(a)$ and $\Phi(b)$ are non-zero in $(I\times \Lambda )^{0}$
and thus $\Phi(ab)=\Phi (a)\, \Phi (b)$ is also non-zero. So $ab
\in S^{(i)}$.

(6) Suppose $a \in S_i \backslash S_{i+1}, b \in S_j \backslash
S_{j+1}$, $ab \notin K$ and $S_i \backslash S_{i+1}$ and $S_j
\backslash S_{j+1}$ are not in a same stem. By part (4), there
exists some $k$ such that $ab \in S^{(k)}$ and $S_k \backslash
S_{k+1}$ is a root. Hence $ab \notin F'_{S_{k} / S_{k+1}}(S /
S_{k+1})$. As by Lemma~\ref{complete}.(\ref{complete1}) $F'_{S_{k} / S_{k+1}}(S / S_{k+1})$ is an ideal of
$S / S_{k+1}$, it follows that $a, b \notin F'_{S_{k} / S_{k+1}}(S
/ S_{k+1})$. Hence, $a, b \in F_{S_{k} / S_{k+1}}(S / S_{k+1})$.
But this contradicts with the assumption that $a$ and $b$ do not
belong to a same stem. This proves one implication of (6). The
converse easily can be verified.
\end{proof}

We now give several consequences of  Theorem~\ref{principal
series}. First we extend Lemma~\ref{RN Groups} as follows.

\begin{cor}
A finite monoid $S$ is \B\ if and only if it is nilpotent.
\end{cor}
\begin{proof} Suppose $S$ is a \B\ finite monoid.
From Theorem~\ref{principal series} we know that the set
consisting of the isolated vertices is the largest nilpotent ideal
$K$ of $S$. Clearly $1\in K$. Hence, $S=K$ is nilpotent. The
result follows.
\end{proof}

Note that in general $(S_i \backslash S_{i+1})(S_j \backslash
S_{j+1}) \nsubseteq K$ does not imply that if $S_i \backslash
S_{i+1}$ is contained in a stem $S^{(h)}$ then $S_j \backslash
S_{j+1} \subseteq S^{(h)}$. However, we can prove the following.

\begin{cor}\label{edge stem}
Let $S$ be a \B\ finite semigroup and $a,b \in S$. If there exists
an edge in ${\mathcal N}_S$ between $a$ and $b$, then there exists
a stem $S^{(i)}$ such that $a, b \in S^{(i)}$.
\end{cor}
\begin{proof}

Since $a$ and $b$ are not isolated vertices, by
Theorem~\ref{principal series}.(\ref{ps1}), both $a$ and $b$ do
not belong to $K$. Let $S_i / S_{i+1}$ and $S_j / S_{j+1}$ be
principal factors of $S$ such that $a\in S_i \backslash S_{i+1}$
and $b \in S_j \backslash S_{j+1}$. If $ab$ or $ba$ is not in $K$,
then $(S_i \backslash S_{i+1}) (S_j \backslash S_{j+1}) \nsubseteq
K$ or $(S_j \backslash S_{j+1}) (S_i \backslash S_{i+1})
\nsubseteq K$ and thus by Theorem~\ref{principal
series}.(\ref{ps6}) the statement obviously holds.

Now suppose that $ab, ba \in K$. Since $\langle a, b\rangle$ is
not nilpotent, by Lemma~\ref{finite-nilpotent}, there exists a
positive integer $m$, distinct elements $x, y\in \langle a,
b\rangle$ and $w_{1}, w_{2}, \ldots,$ $w_{m}\in \langle a,
b\rangle^{1}$ such that $x = \lambda_{m}(x, y, w_{1}, w_{2},
\ldots, w_{m})$, $y = \rho_{m}(x, y, w_{1}, w_{2}, \ldots,$ $
w_{m})$. As $S$ is \B\ and $\langle a\rangle$ and $\langle
b\rangle$ are nilpotent, we get that $\{x, y\} \nsubseteq \langle
a\rangle$ and $\{x, y\} \nsubseteq \langle b\rangle$. Since $ab,
ba \in K$ and $K$ is an ideal, it follows that $x, y \in K$. As
$K$ is nilpotent, Theorem~\ref{principal series}.(\ref{ps1}) then
implies that $\langle x, y\rangle$ is nilpotent, a contradiction.
%
\end{proof}

\begin{cor} \label{connectivity}
Let $S$ be a \B\ finite semigroup. The following properties hold.
\begin{enumerate}
    \item \label{co1} Every stem is connected and any two distinct elements of a stem are connected by a path of length at most $2$.
    \item \label{co2} If $K$ is empty, then $S= S^{(i)}$ for some root $S_i \backslash S_{i+1}$,
    ${\mathcal N}_{S}$ is connected and every two distinct vertices are connected by a path of length at most $2$.
    \item \label{co3} If $S$ does not have any connections, then every
    connected component with more than one element is a stem and it is a subsemigroup.
    \item \label{co4} The union of two stems that have a connection is a connected subset of $\mathcal{N}_{S}$. Furthermore, every shortest path
    in this union has length at most $4$.
\end{enumerate}
\end{cor}

\begin{proof} We use the same notation as in Theorem~\ref{principal series}.

(1) Assume $S_i \backslash S_{i+1}$ is a root and suppose $s, t
\in S^{(i)}$, with $s\neq t$. By Lemma~\ref{RN Rees semigroups},
$S_i \backslash S_{i+1} = \mathcal{M}(G, I, \Lambda; P)$, a
regular Rees matrix semigroup,
with $\abs{I}+\abs{\Lambda}>2$, $G$ is a nilpotent group and all
entries of $P$ are non-zero. By
Lemma~\ref{complete}.(\ref{complete2}), there exists a \B\
homomorphism
$$\Phi: (S / S_{i+1})^{0}  \rightarrow (I\times \Lambda)^{0}.$$
Since $s, t \in S^{(i)}$, we have $\Phi(s)\neq \theta, \Phi(t)
\neq \theta$.

If $\Phi(s) \neq \Phi(t)$, then by
Lemma~\ref{complete}.(\ref{complete4}), there is an edge in
${\mathcal N}_S$ between $s$ and $t$. If $\Phi(s) = \Phi(t)$,
there exists $(c,d) \in I\times \Lambda$ such that $(c,d) \neq
\Phi(s)$ and $g\in G$ such that there is an edge in ${\mathcal
N}_S$ between $s$ and $(g;c,d)$ and between $t$ and $(g;c,d)$.
Hence a shortest path between $s$ and $t$ has length at most $2$.

(2) Assume $K =\emptyset$. Then, by Theorem~\ref{principal
series}.(\ref{ps4}), every element of $S$ belongs to a stem. By
Theorem~\ref{principal series}.(\ref{ps6}), we also get that $S$
has only one stem. Part (1) thus yields that $S$ is connected and
a shortest path between any two distinct elements $s,t\in S$ has
length at most $2$.

(3) Assume that $S$ does not have any connections. Suppose
$S^{(i)}$ is a stem, $s \in S^{(i)}$ and $t \notin S^{(i)}$ and
assume there is an edge in ${\mathcal N}_S$ between $s$ and $t$.
Corollary~\ref{edge stem} implies that there exists a stem
$S^{(k)}$ such that $s,t \in S^{(k)}$. Let $p$ be such that $s \in
S_p \backslash S_{p+1}$ and $S_p /S_{p+1}$ is a principal factor.
Then $S_p \backslash S_{p+1}$ is a connection between $S^{(i)}$
and $S^{(k)}$. This contradicts with the assumption that $S$ does
not have any connections. It follows that any connected component
with more than one element is contained in a stem. Because of part
(1) we actually get that such a connected component is a stem.
Furthermore, by Theorem~\ref{principal series}.(\ref{ps5}), a stem
is a subsemigroup.

(4) Suppose $S_k \backslash S_{k+1}$ is a connection between two
stems $S^{(i)}$ and  $S^{(j)}$. Let  $s, t\in S^{(i)} \cup
S^{(j)}$ and let $x \in S_k \backslash S_{k+1}$. Since $s$ and $x$ belong to the same stem, by part (1) they are connected by a path of
length at most 2. By the same reson $t$ and $x$ are connected by a path of length at most 2. Therefore the result follows.

\end{proof}


\begin{cor} Let $S$ be a \B\ finite semigroup. The following
properties hold.
\begin{enumerate}
 \item Every connected component of ${\mathcal N}_S$ that has more
than one element is a union
of some stems.
 \item If $C$ is  a connected
component of $S$ then $C\cup K$ is a semigroup.
 \item If $C_1, \ldots , C_n$  are the connected components of $S$ with more
than one element then $S/K= C_1^{\theta} \cup \cdots \cup
C_{n}^{\theta}$, a $0$-disjoint union.
\end{enumerate}
\end{cor}

\begin{proof} This follows at once from Theorem~\ref{principal series} and Corollary~\ref{connectivity}.
\end{proof}
%
%
%
So we have shown that every non-isolated connected component of a
\B\ finite semigroup $S$ is a union of stems, say $S_{1},\ldots,
S_{n}$. Hence, every $S_{i}$ has a connection with $S_{j}$ for
some $i\neq j$. However, $S_{i}$ is not necessarily connected with
every $S_{j}$.
We give an example. For this we recall from~\cite{Jes-shah} that the non-commuting  graph ${\mathcal M}_S$ of a
semigroup  $S$ is the graph whose vertices are the elements of $S$
and in which there is an edge between two distinct vertices $x$
and $y$ if these elements do not commute. By~\cite[Lemma
3.5]{Jes-shah}, if $S$ is a band, then
${\mathcal N}_S = {\mathcal M}_S$.

\begin{figure}
\hrule
\begin{picture}(00,65)(45,20)
\gasset{AHnb=0,Nw=1.25,Nh=1.25,Nframe=n,Nfill=y}
\gasset{ExtNL=y,NLdist=1.5,NLangle= 60}

  \put(-22,55){\mbox{$.............$}}
  \put(95,55){\mbox{$.............$}}

  \node(0)(28,30){$\theta$}
  \node(1)(-10,65){$a_i$}
  \node(2)(5,65){$b_i$}
  \node(3)(20,65){$c_i$}
  \node(4)(35,65){$d_i$}
  \node(5)(50,65){$c_{i+1}$}
  \node(6)(65,65){$d_{i+1}$}
  \node(7)(80,65){$c_{i+2}$}
  \node(8)(95,65){$d_{i+2}$}
  \node(9)(5,45){$f_i$}
  \node(10)(35,45){$f_{i+1}$}
  \node(11)(65,45){$f_{i+2}$}
  \node(12)(22,50){$l_{i+1}$}
  \node(13)(52,50){$l_{i+2}$}

  \drawedge(1,2){}
  \drawedge(1,9){}
  \drawedge(9,3){}
  \drawedge(3,4){}
  \drawedge(3,12){}
  \drawedge(3,10){}
  \drawedge(10,5){}
  \drawedge(5,13){}
  \drawedge(11,5){}
  \drawedge(5,6){}
  \drawedge(11,7){}
  \drawedge(7,8){}
\end{picture}
\hrule
\caption{}
\end{figure}
Let $X_0, X_1, \ldots, X_n$ be semigroups ($n\geq 1$) such that
$X_i = \{ a_i, b_i, f_i, c_i, d_i, \theta \}$, for $0 \leq i \leq
n$ and $X_{i}$ has Cayley table
$$
\begin{tabular}{ c|c c c c c c}
           & $\theta$ & $a_i$      & $b_i$      & $f_i$      & $c_i$      & $d_i$\\ \hline
$\theta$   & $\theta$ & $\theta$   & $\theta$   & $\theta$   & $\theta$   & $\theta$  \\
$a_i$      & $\theta$ & $a_i$      & $b_i$      & $b_i$      & $\theta$   & $\theta$  \\
$b_i$      & $\theta$ & $a_i$      & $b_i$      & $b_i$      & $\theta$   & $\theta$  \\
$f_i$      & $\theta$ & $a_i$      & $b_i$      & $f_i$      & $c_i$      & $d_i$\\
$c_i$      & $\theta$ & $\theta$   & $\theta$   & $d_i$      & $c_i$      & $d_i$\\
$d_i$      & $\theta$ & $\theta$   & $\theta$   & $d_i$      & $c_i$      & $d_i$\\
\end{tabular}
$$
Note that the semigroups $X_i$ are isomorphic to the semigroup $T$ given after Definition~\ref{def-connection}.
Furthermore, $X_i \cap X_{i+2}=\{\theta\}$ for $0\leq i \leq n-2$
and $a_{i+1} = c_i$, $b_{i+1}=d_i$, $X_i \cap X_{i+1}=\{a_{i+1},
b_{i+1}, \theta\}$ for $0\leq i \leq n-1.$

We now define the semigroup
 $$S=\bigcup_{0\leq i \leq n} X_i \; \bigcup \; \{ l_{1}, \ldots, l_n \}$$
(where $l_{1}, \ldots, l_{n}$ are the distinct elements not
belonging to $\bigcup_{0\leq i \leq n} X_i$) with multiplication
such that each $X_{i}$ is a subsemigroup and such that
$X_iX_{i+2}=X_{i+2}X_i=\{ \theta \} $ for $0 \leq i \leq n-2$.
Furthermore, for  $0\leq i \leq n-1$,
$$\begin{tabular}{ c|c c c c c}
          & $f_i$     & $f_{i+1}$ & $c_i$ & $d_i$ & $l_{i+1}$\\ \hline
$f_i$     & $f_i$     & $l_{i+1}$ & $c_i$ & $d_i$ & $l_{i+1}$\\
$f_{i+1}$ & $l_{i+1}$ & $f_{i+1}$ & $c_i$ & $d_i$ & $l_{i+1}$\\
$c_i$     & $d_i$     & $d_i$     & $c_i$ & $d_i$ & $d_i$ \\
$d_i$     & $d_i$     & $d_i$     & $c_i$ & $d_i$ & $d_i$ \\
$l_{i+1}$ & $l_{i+1}$ & $l_{i+1}$ & $c_i$ & $d_i$ & $l_{i+1}$ \\
\end{tabular}$$ and
$\{a_i, b_i\}f_{i+1}= f_{i+1}\{a_i, b_i\} = \{c_{i+1},
d_{i+1}\}f_i= f_i\{c_{i+1}, d_{i+1}\} = \theta$ and, for $1 \leq i
\leq n$, we  also have $l_ix=xl_i=\theta$ for $x \in S \backslash
\{l_i, f_{i-1}, f_i, c_{i-1}, d_{i-1}\}$. We claim that $S$ is \B. We prove this  by contradiction. So assume that
there exist distinct elements $x,y\in S$, elements $w_{1}, \ldots , w_{m}\in S^{1}$,  an ideal $I$ of
 $\langle x,y,w_{1},\ldots , w_{m}\rangle$ such that $$\lambda_{t}(x, y, w_{1},
w_{2}, \ldots, w_{t}) = \lambda_{m}(x, y, w_{1},
w_{2}, \ldots, w_{m}),$$ $$\rho_{t}(x, y, w_{1}, w_{2}, \ldots,
w_{t}) = \rho_{m}(x, y, w_{1}, w_{2}, \ldots,
w_{m}),$$ $\lambda_m \neq \rho_m$, $\lambda_m, \rho_m\not\in I$  and  $\langle \lambda_i, \rho_i \rangle$ is nilpotent in $\langle x,y,w_{1},\ldots , w_{m}\rangle/ I$
for some $0 \leq i \leq m$.
Because the subsemigroups $Y_k=\{\theta, a_k, b_k\}$ are ideals of $S$ for $0 \leq k \leq n$, if $\{\lambda_i, \rho_i, w_{i+1}\} \cap Y_k \neq \emptyset$ for some $0 \leq k \leq n$, then $\{\lambda_{i+1}, \rho_{i+1}\} \subseteq Y_k$. Since $\lambda_m \neq \rho_m$, $\{\lambda_{i+1}, \rho_{i+1}\}= \{a_k, b_k\}$ and thus $a_k, b_k\in S \backslash I$.  Also since $\lambda_{i+1}, \rho_{i+1} \in  \langle x,y,w_1, \ldots , w_m\rangle$, we have that  $a_k, b_k \in  \langle x,y,w_1, \ldots , w_m\rangle.$ Therefore there is an edge between $(a_k,b_k)$ in ${\mathcal N}_{\langle x,y,w_1, \ldots, w_m\rangle/I}$.

We claim that $a_k \not \in S(S\backslash \{a_k\})$. Indeed, suppose the contrary, i.e. assume $\alpha \beta= a_k$ with $\alpha \in S$ and
$\beta \in (S\backslash \{ a_{k} \})$. Then $\alpha \beta \beta= a_k \beta$. Since $S$ is band, $a_k=\alpha \beta=a_k\beta$. Now as $a_k (S\backslash \{a_k\})=\{\theta, b_k\}$ we get that $\beta=a_k$, a contradiction. This proves the claim.
Now as $a_k \in\{\lambda_{i}w_{i+1}\rho_{i},\rho_{i}w_{i+1}\lambda_{i}\}$ and $a_k \not \in S(S\backslash \{a_k\})$,
$a_k \in \{\lambda_{i}, \rho_{i}\}$.
Suppose that $\lambda_i = a_k$. Since there is no edge between $\lambda_i$ and $\rho_i$, $\rho_i \in S\backslash \{a_k, b_k, f_k, f_{k-1}, l_{k}\}$, because if $\rho_i \in \{b_k,f_{k-1}, f_{k}, l_{k}\}$ then
there are edges between $\lambda_i$ and $\rho_i$ in ${\mathcal N}_{\langle x,y,w_1, \ldots, w_m\rangle/I}$. Now as $a_kw_{k+1} \in Y_k$ and $Y_k (S\backslash \{a_k,$ $ b_k, f_k, f_{k-1}, l_{k}\})= \{\theta\}$, $\theta \in \{\lambda_{i+1}, \rho_{i+1}\}$, a contradiction. Hence $\{\lambda_i, \rho_i, w_{i+1}\} \subseteq \{f_0, \cdots, f_n, l_1, \cdots, l_n\}$.

Since $\{f_0, \cdots, f_n, l_1, \cdots, l_n\}l_k=l_k\{f_0, \cdots, f_n, l_1, \cdots, l_n\}=\{\theta, l_k\}$ for every $1 \leq k \leq n$, $\{\lambda_i, \rho_i, w_{i+1}\} \subseteq \{f_0, \cdots, f_n\}$. Also $\{f_0, \cdots, f_{k-1}, f_{k+1}, \cdots, f_n\}f_k=$ $f_k\{f_0, \cdots,$ $f_{k-1}, f_{k+1}, \cdots, f_n\}=\{\theta, l_k\}$ for every $1 \leq k \leq n$. Thus $\{\lambda_{i+1}, \rho_{i+1}\}= \{\theta, l_k\}$.
Then $S$ is \B.

As $S$ is a band, we
have that ${\mathcal N}_S = {\mathcal M}_S$. The graph ${\mathcal
N}_S$ is depicted in Figure~1. Between $b_0$ and $d_n$ the
shortest path has a length $4+2n$. Between the roots $\{a_i,
b_i\}$ and $\{a_{i+1}, b_{i+1}\}$, there is a connection, but
there is no connection between the roots $\{a_i, b_i\}$ and
$\{a_{i+2}, b_{i+2}\}$.

We introduce a class of \B\ semigroups for which the connectivity
between the stems is transitive.

\begin{defn} \label{strong-pseudo-def}
A \B\ semigroup $S$ is said to be strong \B\ if it satisfies the
following  properties.
\begin{enumerate}
    \item[(H1)] If ${\mathcal N}_S$ has a connection $S_{i}\backslash S_{i+1}$
between two stems $S^{(i_{1})}$ and $S^{(i_{2})}$, then the   \B\
homomorphisms $\Phi_{1}$ and $\Phi_{2}$ with domains
$(S/S_{i_{1}+1})^{0}$ and $(S/S_{i_{2}+1})^{0}$ respectively are
such that $S_{i}\backslash S_{i+1}$ intersects two different \B\
classes of $\Phi_{1}$ and it also intersects two different \B\
classes of $\Phi_{2}$.
 \item[(H2)] If $S_i / S_{i+1}$ and $S_j / S_{j+1}$
 are principal factors of $S$ with $i < j$ and there is an edge in ${\mathcal N}_S$
 between some of their elements ,
 then $S_i \backslash S_{i+1} \subset F_{S_j / S_{j+1}}(S /
 S_{j+1})$.
\end{enumerate}
\end{defn}

Note that property (H1) implies that each connection intersects
non-trivially the different \B\ classes of the \B\ homomorphism
determined by the stem in which it is contained. Hence,
Lemma~\ref{complete}.(\ref{complete4}) easily yields that  
if $S_i \backslash
S_{i+1}$ and $S_j \backslash S_{j+1}$ are different connections
that are in a same stem, then there exists $(s,t) \in {\mathcal
N}_S$ with  $s \in S_i \backslash S_{i+1}, t \in S_j \backslash
S_{j+1}$. Also if $i < j$, property (H2) implies $S_i \backslash
S_{i+1} \subset F_{S_j / S_{j+1}}(S / S_{j+1})$.

An example of a strong \B\ semigroup is the semigroup $R=\{
\theta, a_{1},a_{2},a_{3},a_{4},b_{1},b_{2},b_{3}\}$ with
multiplication table
$$
\begin{tabular}{ c|c c c c c c c c}
         & $\theta$ & $a_1$    & $a_2$    & $a_3$    & $a_4$    & $b_1$    & $b_2$    & $b_3$\\ \hline
$\theta$ & $\theta$ & $\theta$ & $\theta$ & $\theta$ & $\theta$ & $\theta$ & $\theta$ & $\theta$\\
  $a_1$  & $\theta$ & $a_1$    & $a_2$    & $\theta$ & $\theta$ & $a_1$    & $a_2$    & $a_2$\\
  $a_2$  & $\theta$ & $a_1$    & $a_2$    & $\theta$ & $\theta$ & $a_1$    & $a_2$    & $a_2$\\
  $a_3$  & $\theta$ & $\theta$ & $\theta$ & $a_3$    & $a_4$    & $a_3$    & $a_3$    & $a_4$\\
  $a_4$  & $\theta$ & $\theta$ & $\theta$ & $a_3$    & $a_4$    & $a_3$    & $a_3$    & $a_4$\\
  $b_1$  & $\theta$ & $a_1$    & $a_2$    & $a_3$    & $a_4$    & $b_1$    & $b_2$    & $b_3$\\
  $b_2$  & $\theta$ & $a_1$    & $a_2$    & $a_3$    & $a_4$    & $b_1$    & $b_2$    & $b_3$\\
  $b_3$  & $\theta$ & $a_1$    & $a_2$    & $a_3$    & $a_4$    & $b_1$    & $b_2$    & $b_3$\\
\end{tabular}$$
With a  similar proof to the one  given for  the example stated before Definition~\ref{strong-pseudo-def} one shows that  $R$ is \B.
Furthermore, $R$ has two roots $\{a_1, a_2\}$, $\{a_3, a_4\}$, with
respective stems say $R_{1}$ and $R_{2}$. The set $\{b_1,b_2,
b_3\}$ is a connection between $R_{1}$ and $R_{2}$. The sets
$\{b_1\}$ and  $\{b_2, b_3\}$ belong to different \B\ classes
determined by the root $\{a_1, a_2\}$ and  the sets $\{b_1, b_2\}$
and  $\{b_3\}$ belong to different \B\ classes determined by the
root $\{a_3, a_4\}$. Therefore $R$ is a strong \B\ semigroup.

\begin{lem}\label{kh}
Let $S$ be a strong \B\ finite semigroup. The following properties
hold.
\begin{enumerate}
    \item \label{kh1} If there is a  connection between the  stems $S^{(i_1)}$ and  $S^{(i_2)}$ and also between the
    stems $S^{(i_2)}$ and $S^{(i_3)}$, then there is a connection between the stems $S^{(i_1)}$ and
    $S^{(i_3)}$.
    \item \label{kh2} If there is no connection between two stems $S^{(i)}$ and $S^{(j)}$ then these stems belong to different connected
    components of ${\mathcal N}_{S}$.
\end{enumerate}
\end{lem}

\begin{proof}
(1) Suppose $S_{k_1} \backslash S_{k_1+1}$ is a connection between
the stems $S^{(i_1)}$ and $S^{(i_2)}$ and $S_{k_2} \backslash
S_{k_2+1}$ is a connection between the stems $S^{(i_2)}$ and
$S^{(i_3)}$. Suppose that $k_1 < k_2$. Since $S$ is a strong \B\
semigroup, $S_{k_1} \backslash S_{k_1+1} \subseteq F_{S_{k_2} /
S_{k_2+1}}(S/S_{k_2+1})$ and thus $S_{k_1} \backslash S_{k_1+1}
\subseteq S^{(k_2)}$. Since $S_{k_2} \backslash S_{k_2+1}$ $
\subseteq S^{(i_3)}$, we obtain from Theorem~\ref{principal
series}.(\ref{ps2}) that  $S_{k_1} \backslash S_{k_1+1} \subseteq
S^{(i_3)}$. Therefore $S_{k_1} \backslash S_{k_1+1}$ is a
connection between the stems $S^{(i_1)}$ and $S^{(i_3)}$.

(2) Let $S^{(i)}$ and $S^{(j)}$ be two different stems. Suppose
that $x \in S^{(i)}$,  $y \in S^{(j)}$ and that $x=z_{0}, z_1,
\ldots, z_n, z_{n+1}=y$ is a path between $x$ and $y$. Because of
Theorem~\ref{principal series}.(\ref{ps4}), for $0\leq i \leq n$,
we get that there exist stems $S^{(n_{z_i})}$ and subsets
$S_{m_{z_i}} \backslash S_{m_{z_i}+1}$ of $S^{(n_{z_{i}})}$  such
that $z_i \in S_{m_{z_i}} \backslash S_{m_{z_i}+1}$ for a
principal factor $S_{m_{z_i}} / S_{m_{z_i}+1}$.

Suppose that $m_{z_i} < m_{z_{i+1}}$. Since $S$ is a strong \B\
semigroup and there is an edge in ${\mathcal N}_S$ between $z_i$
and $z_{i+1}$, we have $S_{m_{z_i}} \backslash S_{m_{z_i}+1}
\subseteq F_{S_{m_{z_{i+1}}} /
S_{m_{z_{i+1}}+1}}(S/S_{m_{z_{i+1}}+1})$. Hence $S_{m_{z_i}}
\backslash S_{m_{z_i}+1}$ is a connection between the stems
$S^{(n_{z_i})}$ and $S^{(n_{z_i}+1)}$. Similarly we have this
result for $m_{z_i} > m_{z_{i+1}}$.  Therefore by part (1), there
is a connection between $S^{(i)}$ and $S^{(j)}$. This contradicts
with there is no connection between them.
\end{proof}

Note that if $S$ is a \B\ finite semigroup that is not {strong}
\B\ then in general Lemma~\ref{kh} does not hold. For example, the
\B\ semigroup with graph as depicted in Figure~1, does not satisfy
property (H1) and it satisies neither (1) nor (2) of
Lemma~\ref{kh}. An example of
a \B\ finite semigroup that does not satisfy property
(H2) and it satisfies neither (1) nor (2) of
Lemma~\ref{kh}, is the semigroup $T =\{ \theta,
a_{1},a_{2},a_{3},b_{1},b_{2},b_{3},f_{1},f_{2},$ $f_{3},f_{4}\}$
with multiplication table
$$
\begin{tabular}{ c|c c c c c c c c c c c}
         & $\theta$ & $a_1$    & $b_1$    & $a_2$    & $b_2$    & $a_3$    & $b_3$    & $f_1$    & $f_2$    & $f_3$    & $f_4$   \\ \hline
$\theta$ & $\theta$ & $\theta$ & $\theta$ & $\theta$ & $\theta$ & $\theta$ & $\theta$ & $\theta$ & $\theta$ & $\theta$ & $\theta$\\
  $a_1$  & $\theta$ & $a_1$    & $b_1$    & $\theta$ & $\theta$ & $\theta$ & $\theta$ & $b_1$    & $a_1$    & $\theta$ & $\theta$\\
  $b_1$  & $\theta$ & $a_1$    & $b_1$    & $\theta$ & $\theta$ & $\theta$ & $\theta$ & $b_1$    & $a_1$    & $\theta$ & $\theta$\\
  $a_2$  & $\theta$ & $\theta$ & $\theta$ & $a_2$    & $b_2$    & $\theta$ & $\theta$ & $b_2$    & $a_2$    & $b_2$    & $a_2$   \\
  $b_2$  & $\theta$ & $\theta$ & $\theta$ & $a_2$    & $b_2$    & $\theta$ & $\theta$ & $b_2$    & $a_2$    & $b_2$    & $a_2$   \\
  $a_3$  & $\theta$ & $\theta$ & $\theta$ & $\theta$ & $\theta$ & $a_3$    & $b_3$    & $\theta$ & $\theta$ & $b_3$    & $a_3$   \\
  $b_3$  & $\theta$ & $\theta$ & $\theta$ & $\theta$ & $\theta$ & $a_3$    & $b_3$    & $\theta$ & $\theta$ & $b_3$    & $a_3$   \\
  $f_1$  & $\theta$ & $a_1$    & $b_1$    & $a_2$    & $b_2$    & $\theta$ & $\theta$ & $f_1$    & $f_2$    & $b_2$    & $a_2$   \\
  $f_2$  & $\theta$ & $a_1$    & $b_1$    & $a_2$    & $b_2$    & $\theta$ & $\theta$ & $f_1$    & $f_2$    & $b_2$    & $a_2$   \\
  $f_3$  & $\theta$ & $\theta$ & $\theta$ & $a_2$    & $b_2$    & $a_3$    & $b_3$    & $b_2$    & $a_2$    & $f_3$    & $f_4$   \\
  $f_4$  & $\theta$ & $\theta$ & $\theta$ & $a_2$    & $b_2$    & $a_3$    & $b_3$    & $b_2$    & $a_2$    & $f_3$    & $f_4$   \\
\end{tabular}$$
We leave it to the reader to verify that  $T$ is \B. Further, the semigroup $T$ has three roots $\{a_1, b_1\}$, $\{a_2, b_2\}$,
$\{a_3, b_3\}$, with respective stems say $T_{1}$, $T_{2}$ and
$T_{3}$. The set $\{f_1,f_2\}$ is a connection between $T_{1}$ and
$T_{2}$ and the set $\{f_3,f_4\}$ is a connection between $T_{2}$
and $T_{3}$, but there is no connection between $T_{1}$ and
$T_{3}$. As $$\{f_3,f_4\} \subset F'_{\{a_1, b_1, a_2, b_2, a_3,
b_3, f_1,f_2, \theta\} / \{a_1, b_1, a_2, b_2, a_3, b_3,
\theta\}}(T / \{a_1, b_1, a_2, b_2, a_3, b_3, \theta\}),$$ $T$ is
not a strong \B\ semigroup.

\begin{cor} \label{connected component kh}
Let $S$ be a strong \B\ finite semigroup with $K$ the ideal of
isolated vertices. Two elements $x$ and $y$ of $S\backslash K$ are
in the same connected component of ${\mathcal N}_{S}$ if and only if
there exists an element $z \in S$ such that $xz\not\in K$ and
$zy\not\in K$.
\end{cor}

\begin{proof}
Suppose $x$ and $y$ are in $S \backslash K$. Assume $x$ and $y$
are in a same stem, say $S^{(h)}$. Because of
Theorem~\ref{principal series}.(\ref{ps5}), $S^{(h)}$ is a
subsemigroup of $S$ and because of
Corollary~\ref{connectivity}.(\ref{co1}), $S^{(h)}$ is connected.
Hence, by Theorem~\ref{principal series}.(\ref{ps5}), $xy, y^2
\notin K$ and $x$ and $y$ are in the same connected component.

Now assume that $x$ and $y$ are not in a same stem. Because of
Theorem~\ref{principal series}.(\ref{ps4}), there exist stems
$S^{(i)}, S^{(j)}$ such that $x \in S^{(i)}$ and $y \in S^{(j)}$.
We need to deal with two cases:  (i) between $S^{(i)}$ and
$S^{(j)}$, there is a connection, say $S_k \backslash S_{k+1}$,
(ii) between $S^{(i)}$ and $S^{(j)}$ there is no connection.

(i) Let $z\in S_k \backslash S_{k+1}\subseteq S^{(i)}\cap S^{(j)}$.
Since $S^{(i)}$ and $S^{(j)}$ are semigroups we get that $xz\in
S^{(i)}$ and $zy\in S^{(j)}$ and thus $xz\not\in K$ and $zy\not\in
K$. Corollary~\ref{connectivity}.(\ref{co4}) implies that $x$ and
$y$ are in the same connected component.

(ii) Since, by
assumption, there is no connection between $S^{(i)}$ and
$S^{(j)}$, Lemma~\ref{kh}.(\ref{kh2}) yields that  $x$ and $y$ are
in different connected components. Now let $z\in S$ and assume
$xz, zy \not\in K$. Then, by Theorem~\ref{principal
series}.(\ref{ps6}), there exists stems $S^{(k_1)}$ and
$S^{(k_2)}$ such that $x,z \in S^{(k_1)}$ and $z,y \in S^{(k_2)}$.
Corollary~\ref{connectivity}.(\ref{co4}) implies that $x$ and $y$
are in the same connected component, a contradiction.

The result follows.
\end{proof}

\begin{defn}
Let $S$  be a \B\ finite semigroup and let $K$ be its largest
nilpotent ideal. For $a\in S$ put $K_{a}=\{ s\in S \mid as\not\in
K\}$. Because of  Theorem~\ref{principal series}.(\ref{ps6}),
$K_{a}=\{ s\in S \mid sa\not\in K\}$.
\end{defn}

\begin{cor}
Let $S$ be a strong \B\ finite semigroup. The connected components
with more than one element are the maximal elements in the set $\{
K_{a} \mid a\in S\backslash K\}$.
\end{cor}

\begin{proof}
By Theorem~\ref{principal series}, if a connected component of
$a\in S$ has more than one element then $a\in S\backslash K$.
Furthermore, $a\in K_{a}$ for any $a\in S\backslash K$. Let $x\in
S\backslash K$ be such that $K_{x}$ is a maximal element in the
set $\{ K_{a} \mid a\in S\backslash K\}$.  We need to prove that
$K_{x}$ is a connected component.

First we notice that $K_{x}$ is connected. Indeed, if $y \in K_x$,
then $xy \notin K$. Hence, again by Theorem~\ref{principal
series}, $x$ and $y$ are in a same stem. Since every stem is
connected by Corollary~\ref{connectivity}, there is a path between
$x$ and $y$. Hence all the element of $K_x$ are in the same connected component.

Second, suppose that $z\in S$ is in the connected component
containing $x$. It remains to be shown that $z\in K_{x}$. To do
so, we first notice from Corollary~\ref{connected component kh}
that  there exists $v\in S$ such that $xv\not\in K$ and $vz\not\in
K$. Because of Theorem~\ref{principal series}.(\ref{ps6}), there
exist stems $S^{(i_1)}$ and $S^{(i_2)}$ such that $x,v \in S^{(i_1)}$,
$z,v \in S^{(i_2)}$. Assume $x\in S_{n_x}\backslash S_{n_x+1}$,
$z\in S_{n_z}\backslash S_{n_z+1}$ and $v\in S_{n_v}\backslash
S_{n_v+1}$, where  $T_{1}=S_{n_x}/ S_{n_x+1}$, $T_{2}=S_{n_z}/
S_{n_z+1}$ and $T_{3}=S_{n_v}/ S_{n_v+1}$ are principal factors of
$S$. 
As $S$ is strong \B\ and $T_3$ is a connection between the stems
$S^{(i_1)}$ and $S^{(i_2)}$,  $T_3$
intersects non-trivially different \B\ classes of the \B\
homomorphisms determined by these stems. Hence by
Lemma~\ref{complete}.(\ref{complete4}), there is an edge between
some non-zero elements of $T_1$ and $T_3$ and there also is an
edge between some non-zero elements  of $T_2$ and $T_3$ in
${\mathcal N}_S$. Consequently, either $(T_{1}\backslash \{ \theta \})
\subseteq F_{T_3}(S/S_{n_v+1})$ or $(T_3 \backslash \{ \theta \})
\subseteq F_{T_1}(S/S_{n_x+1})$, and either $(T_2 \backslash \{ \theta \})
\subseteq F_{T_3}(S/S_{n_v+1})$ or $(T_3 \backslash \{ \theta \})
\subseteq F_{T_2}(S/S_{n_z+1})$. We therefore need to deal with
four cases.

(Case 1) $(T_1 \backslash \{ \theta \})  \subseteq
F_{T_3}(S/S_{n_v+1})$ and $(T_3 \backslash \{ \theta \}) \subseteq
F_{T_2}(S/S_{n_z+1})$.\\
By Theorem~\ref{principal series}.(\ref{ps2}), we get that $(T_1
\backslash \{\theta \})\subseteq F_{T_2}(S/S_{n_z+1})$. Hence,
$(T_1 \backslash \{ \theta \})$ $( T_2 \backslash \{ \theta \})
\subseteq S^{(i_2)}$.  So, by
Theorem~\ref{principal series}.(\ref{ps6}), $T_{1}\backslash \{
\theta \}$ and $T_{2}\backslash \{ \theta \}$ are in a same stem.
Consequently, by Theorem~\ref{principal series}.(\ref{ps5}),
$xz\not\in K$ and thus $z\in K_{x}$.

(Case 2) $(T_1 \backslash \{ \theta \}) \subseteq
F_{T_3}(S/S_{n_v+1})$ and $(T_2 \backslash \{ \theta \}) \subseteq
F_{T_3}(S/S_{n_v+1})$.\\
In this case, by Theorem~\ref{principal series}.(\ref{ps2}),
$x,z\in (T_{1}\backslash \{ \theta \})\cup (T_{2}\backslash \{
\theta \}) \subseteq S^{(i_2)}$. Hence, by Theorem~\ref{principal
series}.(\ref{ps5}), $xz\not\in K$. So again $z\in K_{x}$.

(Case 3) $(T_3\backslash \{ \theta \}) \subseteq
F_{T_1}(S/S_{n_x+1})$ and $ (T_2 \backslash \{ \theta \})
\subseteq
F_{T_3}(S/S_{n_v+1})$.\\
As in (Case 1) one obtains that $z \in K_x$.

(Case 4) $(T_3\backslash \{ \theta\}) \subseteq
F_{T_1}(S/S_{n_x+1})$ and $( T_3\backslash \{ \theta \}) \subseteq
F_{T_2}(S/S_{n_z+1})$.\\
Clearly $x,z \in K_v$. Assume $y\in K_x$ (and thus, in particular,
$y\not\in K$) and $y\in S_{n_y}\backslash S_{n_y+1}$ for some
principal factor $T_{4}=S_{n_y}/ S_{n_y+1}$.

We claim that $y\in K_{v}$. Let $S^{(i_3)}$ be a stem such that $x,y \in S^{(i_3)}$. If $S^{(i_3)}=S^{(i_1)}$, then by Theorem~\ref{principal series}.(\ref{ps5}), $yv \not \in K$ and thus $y \in K_v$. Suppose that $S^{(i_3)}\neq S^{(i_1)}$. Then $T_1$ is a connection
between $S^{(i_1)}$ and $S^{(i_3)}$. Since $S$ is strong \B,  we know that
 either $(T_4 \backslash
\{ \theta \} ) \subseteq F_{T_1}(S/ S_{n_x+1})$ or $(T_1
\backslash \{ \theta \} ) \subseteq F_{T_4}(S/ S_{n_y+1})$. If
$(T_4 \backslash \{ \theta \}) \subseteq F_{T_1}(S/ S_{n_x+1})$
then since $(T_3 \backslash \{ \theta \}) \subseteq F_{T_1}(S/ S_{n_x+1})$, as in Case 2 we obtain that $vy \not \in K$ and thus $y \in K_v$. If $(T_1
\backslash \{ \theta \}) \subseteq F_{T_4}(S/ S_{n_y+1})$ then,
since $(T_3 \backslash \{\theta \} ) \subseteq
F_{T_1}(S/S_{n_x+1})$, as in Case 1 we obtain that $vy \not \in K$ and thus $y \in K_v$.
This finishes the proof of
the claim.

So we have proved that  $K_x\subseteq K_v$. As, by assumption,
$K_x$ is a maximal element in the set $\{ K_{a} \mid a\in
S\backslash K\}$, we get that $K_x = K_v$. Since $z\in K_{v}$ we
thus obtain that $z \in K_x$, as desired.
\end{proof}

{\bf Acknowledgement} The authors would like to thank the referee for an exhaustive and detailed report that resulted in an improved paper.


\end{document}